\documentclass[11pt,reqno]{amsart}
\usepackage{latexsym}
\usepackage{amssymb}
\usepackage{amsmath}
\usepackage{amscd}
\usepackage{amsthm}
\usepackage{graphicx}
\usepackage{color}
\usepackage{bm}

\newcommand\dela[1]{}
\allowdisplaybreaks

\def\XXint#1#2#3{{\setbox0=\hbox{$#1{#2#3}{\int}$ }
\vcenter{\hbox{$#2#3$ }}\kern-.6\wd0}}

\newtheorem{theorem}{Theorem}
\numberwithin{equation}{section}
\numberwithin{theorem}{section}
\newtheorem{lemma}[theorem]{Lemma}

\numberwithin{equation}{section}

\newcommand{\wt}{\widetilde}
\newcommand{\e}{\varepsilon}

\def\P{\mathbb{ P}}
\def\E{\mathbb{E}}
\def\P1{\mathbb{P}^{1}}

\def\o{\overline}
\begin{document}


\title{Homogenization of Brinkman flows in heterogeneous dynamic media}

\author[H. Bessaih]{Hakima Bessaih}
\address{University of Wyoming, Department of Mathematics, Dept. 3036, 1000
East University Avenue, Laramie WY 82071, United States}
\email{ bessaih@uwyo.edu}

\author[Y. Efendiev]{ Yalchin Efendiev}
\address{Department of Mathematics \& ISC,
Texas A\&M University,
Numerical Porous Media SRI Center, CEMSE Division,
King Abdullah University of Science and Technology, 
Thuwal 23955-6900,
Kingdom of Saudi Arabia}
\email{efendiev@math.tamu.edu}

\author[F. Maris]{ Florian Maris}
\address{Numerical Porous Media SRI Center, CEMSE Division,
King Abdullah University of Science and Technology, 
Thuwal 23955-6900,
Kingdom of Saudi Arabia}
\email{florinmaris@gmail.com}

\maketitle


{\footnotesize
\begin{center}


\end{center}
}
\begin{abstract}
In this paper, we study Brinkman's equations with microscale properties
that are highly heterogeneous in space and time. The time variations are
controlled by a stochastic particle dynamics described by an SDE. 
The particle dynamics can be thought as particle deposition
that often occurs in filter problems. Our main results include
the derivation of macroscale equations and showing that the macroscale
equations are deterministic. The latter is important for our (also many
other) applications as it greatly simplifies the macroscale equations.
We use the asymptotic properties of the SDE and the periodicity of the Brinkman's coefficient in the space variable to prove the convergence result. The SDE has a unique invariant measure that is ergodic and strongly mixing. The macro scale equations are derived through an averaging principle of the slow motion (fluid velocity) with respect to the fast motion (particle dynamics) and also by averaging the Brinkman's coefficient with respect to the space variable. Our results can be extended to more general nonlinear diffusion equations with heterogeneous coefficients. 
\end{abstract}

\maketitle

\section{Introduction and formulation of the problem}
\label{sec1}

\subsection{Motivation.}
\label{subs11}

In many porous media application problems, the media is subject to change due
to pore-scale processes. For example, in filter applications \cite{galina, il04, gk10, ikls11}, 
the media
properties and microscale geometry change due to particles 
are captured by the filter (see Figure \ref{schematic} for illustration).
In this figure, we depict a filter element and particle deposition process
(following \cite{galina}).
The particle deposition changes the microscale
geometry of the filter and thus can greatly affect its macroscopic properties
that are used in simulations \cite{il04, galina}. 
The change due to particle deposition is described
by Stochastic Differential Equations (SDEs) \cite{galina} 
where the particles' mean
velocities are affected by the fluid velocity. Thus, the modified effective
properties strongly depend on particle dynamics and deriving and understanding
these effective properties are essential for many of these applications.
Motivated by this application, we consider a Brinkman model (cf. \cite{galina, il04})  where the permeability changes due to particle dynamics that are driven by the fluid flow.

\begin{figure}[tb]
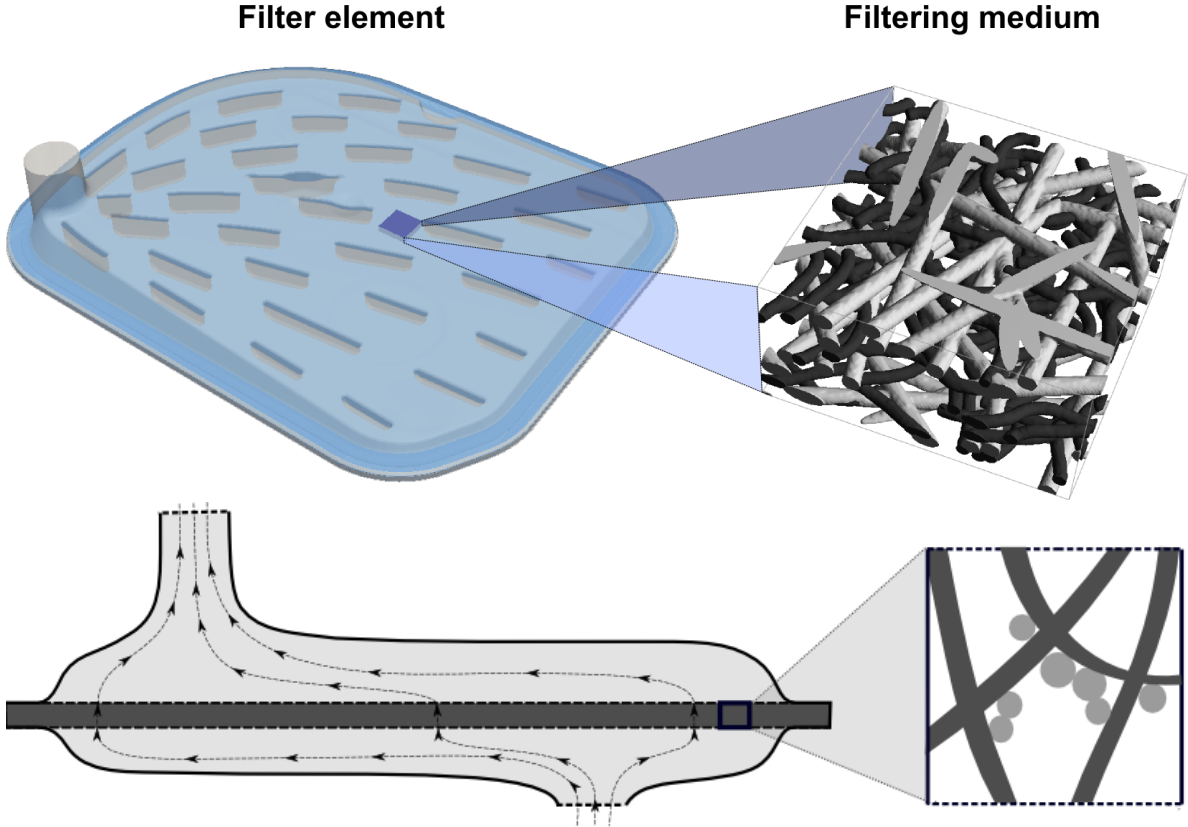

  \centering
 \includegraphics[width=1.0 \textwidth]{micromacro}
  \includegraphics[width=1.0 \textwidth]{mima_new}
  \caption{Illustration of a filter element (top) and a cross section of a filter with a particle deposition (bottom). Courtesy of G. Printsypar.}
  \label{schematic}
\end{figure}

In the paper, we derive macroscopic model assuming that the particle dynamics at microscale
occurs in a much faster time scale compared to the flow. This is typical 
in these applications due to the particle dynamics and their interaction. 
We derive
a macroscopic model where the new upscaled permeability is computed
using spatial microscale variations of the permeability and fast dynamics. 
Besides computing the permeability value, we show that the permeability is deterministic.
This is a useful findings as it allows to compute the upscaled permeability
in a deterministic manner and avoid stochastic macroscale PDEs.

Even though our application is specific for Brinkman' equations, 
our mathematical concepts can be used for many important
applications where the media properties at the microscale are affected
by SDE. An example can be a diffusion equation with heterogeneous coefficients
that depend on a field described by SDE. In general, we have
\begin{equation}
\begin{split}
\dfrac{\partial u^\e}{\partial t} (t,x)= L(a({x\over\e},v^\e),u^\e)\\
d v^\e(t,x) = -\dfrac{1}{\e} (v^\e(t,x)-u^\e(t,x)) dt+ \sqrt {\dfrac{Q}{\e}} dW(t,x),
\end{split}
\end{equation}
where  $W(t)$ is a  standard Brownian motion. For example, 
$L(a({x\over\e},v^\e),u^\e)=\nabla\cdot(a({x\over\e},v^\e)\nabla u^\e)$.
The question of interest is to derive macroscale equations which will be 
investigated in our future works.

\subsection{Mathematical model}
We consider the following system:
\label{subs12}
\begin{equation}
\label{system}
\left\{
\begin{array}{rll}
\dfrac{\partial u^\e}{\partial t} (t,x)&= \Delta u^\e(t,x) - \nabla p^\e(t,x) - \alpha \left(\dfrac{x}{\e}, v^\e(t,x)\right) u^\e(t,x) + f(t,x) &\mbox{ in }\ [0,T] \times D , \\
\operatorname{div}u^\e(t,x) &=0 &\mbox{ in }\ [0,T] \times D, \\
d v^\e(t,x) &= -\dfrac{1}{\e} (v^\e(t,x)-u^\e(t,x))dt + \sqrt {\dfrac{Q}{\e}} dW(t,x) &\mbox{ in }\ [0,T] \times D, \\
u^\e(t,x) &=0 &\mbox{ on }\  [0,T]\times\partial D, \\
u^\e(0,x) & = u_0^\e(x)&\mbox{ in }\  D,\\
v^\e(0,x) & = v_0^\e(x)&\mbox{ in }\  D,
\end{array}
\right.
\end{equation}
where $D$ is a bounded domain of $\mathbb{R}^3$ with a smooth boundary $\partial{D}$, $u^\e$ and $p^\e$ are the velocity and the pressure of the fluid and $v^\e$ the velocity of a particle. $\alpha$ is the Brinkman coefficient and describes the porosity of the medium and is affected by the particle deposition as explained in the previous section. $W(t)$ is an  $L^2(D)^3$-valued  standard Brownian motion defined on a complete probability basis $(\Omega, \mathcal{F}, \mathcal{F}_t, {\mathbb P})$ with expectation $\E$,   and $Q$ is a bounded linear operator on $L^2(D)^3$ of trace class. $u^\e_0$ and 
$v^\e_0$ are the initial velocities for the fluid and for the particle, and $f$ is an external force.

The first equation in system \eqref{system} is a Brinkman type equation. 
For a fixed $\epsilon$ and when the equation for $u^\e$ is not coupled with the equation for $v^\e$, number of papers have been devoted to its mathematical analysis see e.g. \cite{varga} and the references therein. In \cite{varga} the authors study the existence, uniqueness and regularity of solutions for the 
Brinkman-Forchheimer equation in dimension 3 and their global attractors. Let us observe that in \cite{varga}, the term $\alpha(u)u$ is assumed to be monotone, while it's not in the current paper. In particular, the authors in \cite{varga} studied the polynomial case thoroughly. Their result was extended to the convective Brinkman-Forchheimer equation which includes the nonlinear term of the Navier-Stokes equations. Similar results are obtained in \cite{saber} where  a dissipative term is added to the equation.

Our main goal in this paper is to study the asymptotic behavior of the solutions of system  \eqref{system} when  $\e\longrightarrow 0$. Notice that $u^\e$ is random through the function $\alpha$ that depends on  the stochastic process $v^\e$ solution of a stochastic differential equation. Moreover, the function $\alpha$ is a multi scale function.  Here, $u^\e$ is the slow component and $v^\e$ is the fast one.  We will prove that $u^\e$ converges to an averaged velocity $\o{u}$ solution of the averaged equation  \eqref{eqou} where the averaged operator $\o\alpha$ is given by \eqref{defoalpha}. Here, the averages are taken with respect to the periodic variable $y$ and the invariant measure associated to the process $v^\e$ for a frozen $u^\e$.


There is a quite large number of papers dealing with averaging principles for finite dimensional systems in both deterministic and stochastic systems.
Less has been done in the infinite dimensional setting, we refer to \cite{ Cerrai2009, CF2008} and the references therein. There is not much in the literature dealing with averaging systems in porous media. 
In \cite{CF2008}, the authors prove an averaging principle for a very general class of stochastic PDEs. 
Our system looks similar to theirs with a very important difference. Our function $\alpha(\cdot, v) u$ is not Lipchitz and it contains the variable $x/\e$ that describes the heterogeneities of the medium. 
Hence, their results although very general can't be used here. 

For $\e>0$ fixed, the well posedness of system \eqref{system} does not follow from classical results and has to be studied accordingly. In this paper, we assume that $\alpha\in C^{1}_{b}$. Despite this condition, the term  $\alpha(\cdot, v) u$ in system \eqref{system} is neither Lipschitz nor autonomous. 

We prove the existence of weak solutions by using a Galerkin approximation $(u_{n}^\e, v_{n}^\e)$
that is solution of a well posed system and then pass to the limit on $n$ after performing some uniform estimates in $n$. These estimates are also uniform in $\e$. 
By using our assumption on $\alpha$ and the special form of our system, we are able to prove the uniqueness of the weak solution  $(u^\e, v^\e)$. We prove that our weak solution is also strong, and get better uniform estimates in $\e$ for the solution $u^\e$ in the Sobolev space $H^{2}(D)^3$. 

Then, we study the asymptotic behavior of the fast motion variable $v^\e$ for a frozen slow motion variable $u^\e$. Indeed, we consider the SDE \eqref{v} for a given $\xi$. It has a mild solution which is also a strong solution. Its transition semigroup $P_t^\xi$ is well defined and has a unique invariant measure 
$\mu^{\xi}$ which is ergodic and strongly mixing. 
We will consider the auxiliary equation 

\begin{equation*}
c(\e) \Psi^\e(\eta,\xi) - \mathcal{L}^\xi \Psi^\e(\eta,\xi) =  \int_D \left(\alpha^\e(\eta) - \o{\alpha^\e}(\xi)\right)\xi\phi dx.
\end{equation*}
for $c(\e) >0$ to be chosen later.  The solution is given by

\begin{equation*}
\Psi^\e (\eta,\xi) = \int_0^\infty e^{-c(\e) t} P_t^\xi \left[ \int_D \left(\alpha^\e(\cdot) - \o{\alpha^\e}(\xi)\right)\xi\phi dx (\eta)\right] dt.
\end{equation*}

Here $\mathcal{L}^\xi$ is the generator of the semigroup $P_t^\xi$, $\xi, \eta \in L^2(D)^3$ and
$\phi$ is an arbitrary function in $\bold{V}$. The operator $ \alpha^\e$ is defined in section 3 while the operator 
$\o{\alpha^\e}$ is defined in section 5 and refers to the averaged of $\alpha^\e$ wrt to the invariant measure 
$\mu^{\xi}$. The main difficulty stands in passing  to the limit on the term 
\begin{equation}\label{term}
\int_D \left( \alpha^\e(v^\e(t)) u^\e(t) - \o{\alpha}(\o{u}(t)) \o{u}(t) \right)\phi dx
\end{equation}
for $\phi \in \bold{V}$, where $\o\alpha$ is defined in \eqref{defoalpha} at the beginning of section 5. 
This is done by using the It\^o formula on $\Psi^\e (u^\e,v^\e)$ and isolating the term  \eqref{term}. We mainly follow the idea already introduced  in \cite{CF2008}. By using the uniform estimates obtained in Section 3, a tightness argument and some known results for periodic functions, see \cite{A-2s} (lemma 1.3) the passage to the limit is performed in distribution. We obtain a  convergence in probability by using the fact that the limit $\o{u}$ is deterministic.

The paper is organized as follows, Section 2 is dedicated to the introduction of the functional setting and assumptions. In section 3, system \eqref{system} is analyzed for every $\e>0$. In particular existence of strong solutions are established with their uniqueness and their uniform estimates with respect to $\e$.   The fast motion $v^\e$ is analyzed in section 4, these are known results and we give some references. The passage to the limit is performed in section 5. Furthermore, the well posedness of the averaged equation is established.

\section{Preliminaries and Assumptions}
\label{subs13}
We make the following notations for the spaces that will be used throughout the paper. 
For any Hilbert space $X$, $B_b(X)$ denotes the Banach space of bounded Borel functions 
$\phi : X\to \mathbb{R}$ endowed with the supremum norm:
$$\|\phi\|_{B_b(X)} = \sup_{x\in X} |\phi (x)|.$$
$C_b(X)$ denotes the subspace of bounded and continuous functions and $C_b^k(X)$ the subspace of functions that are $k$ times Fr\^{e}chet differentiable with continuous and bounded derivatives up to order $k$ endowed with the norm:
$$\|\phi\|_{C^k_b(X)} =\max_{0\leq l \leq k} \sup_{x\in X} \left\|\dfrac{\partial ^l\phi} {\partial x^l} (x)\right\|_{\mathcal{L}^l( X)},$$
where $\mathcal{L}^0(X)=\mathbb{R}$ and for every $i>0$, 
$\mathcal{L}^i(X)=\mathcal{L}(X;\mathcal{L}^{i-1}(X))$, the Banach space of bounded linear operators from $X$ to $\mathcal{L}^{i-1}(X)$.
For $Y=[0,1]^3$ the space $C_\#(Y)$ denotes the space of continuous functions on $Y$ that are $Y$-periodic and the space $L^2_\#(Y)$ denotes the closure of $C_\#(Y)$ in $L^2(Y)$.

We denote by $\bold{H}$ and $\bold{V}$ respectively the closures of $\mathcal{V}$ in $L^2(D)^3$ and $H^1(D)^3$ where 
\begin{equation}
\label{spaceV}
\mathcal{V}:=\left\{u\in C^\infty\left(\overline{D}\right)^3 \ \ |\ \ \mbox{div }u=0,\ \  u=0 \mbox{ on } \partial D\right\}.
\end{equation}
$\bold{H}$ is a Hilbert space with inner product inherited from  $L^2(D)^3$ denoted by $\langle\cdot,\cdot\rangle$. 
Denoting by $\bold{H} '$ and $\bold{V}'$ the dual spaces, if we identify  $\bold{H}$ with  $\bold{H} '$
then we have the Gelfand triple $\bold{V}\subset \bold{H} \subset \bold{V}'$ with continuous injections.  The dual pairing between $\bold{V}$ and $\bold{V}'$ will be also denoted by $\langle\cdot,\cdot\rangle$. 
The norm in any space $X$ will be denoted by $\|\cdot\|_{X}$


The function $\alpha$ is positive and satisfies the following conditions:
\begin{equation}
\label{alpha1}
\alpha (\cdot, x) \in L^2_\#(Y),
\end{equation}
for every $x\in \mathbb{R}^3$.
\begin{equation}
\label{alpha2}
\alpha (y, \cdot) \in C_b^1(\mathbb{R}^3),
\end{equation}
for every $y \in Y$, with 
\begin{equation}
\label{alpha3}
\|\alpha (y, \cdot)  \|_{C_b^1(\mathbb{R}^3)} \leq C,
\end{equation}
independent of $y\in Y$. We remark that condition \eqref{alpha3} implies uniform boundedness for $\alpha$ on $Y\times \mathbb{R}^3$ as well as Lipschitz condition in the second variable, uniform with respect to the first one.

We also assume that $f\in L^2(0,T;L^2(D)^3)$.



\section{Study of the system}
\label{sec2}
In this section we prove the existence and uniqueness of the solution for the system \eqref{system} as well as some uniform estimates.

\subsection{Well-posedness of the system \eqref{system}}
For any $\e>0$ we denote by $\alpha^\e$ the operator,
\begin{equation}\label{alpha}
\alpha^\e: L^2(D)^3 \to L^\infty(D),\ \ \alpha^\e(\eta) (x) = \alpha\left(\dfrac{x}{\e},\eta(x)\right).
\end{equation}
Let us show that $\alpha^\e$ is a well defined operator. Given the condition \eqref{alpha2}, we need only to show the measurability in $x$ of $\alpha^\e(\eta)$ for any $\eta \in L^2(D)^3$. For such a function, we consider a sequence $\eta_n \in C_0(D)$ convergent to $\eta$ pointwise in $D$. The function $(y,x) \to \alpha (y,\eta_n(x))$ is a Carath\'{e}odory function, measurable in $y$ and continuous in $x$, so $x\to \alpha\left(\dfrac{x}{\e} , \eta_n(x) \right)$ is measurable, and by the Lipschitz condition of $\alpha$ is pointwise convergent to $\alpha^\e(\eta)$, which shows that $\alpha^\e(\eta)$ is measurable. Moreover

\begin{theorem}\label{thexun}
Assume that $u^\e_0\in \bold{H}$ for every $\e>0$, then for each $T>0$, there exists a unique solution of the system \eqref{system}, $u^\e \in L^\infty(\Omega;C([0,T];\bold{H})\cap L^2(0,T;\bold{V}))$ and $v^\e \in L^2(\Omega;C([0,T];L^2(D)^3)$ in the following sense: $\mathbb{P}$ a.s.
\begin{equation}
\label{weaksole}
\int_D u^\e(t) \phi dx - \int_D u^\e_0 \phi dx + \int_0^t \int_D \nabla u^\e (s) \nabla \phi dx ds +\int_0^t \int_D \alpha^\e(v^\e) u^\e \phi dx ds = \int_0^t \int_D f(s) \phi dx ds,
\end{equation}
for every $t\in[0,T]$ and every $\phi \in \bold{V}$, and
\begin{equation}
\label{mildsole}
v^\e(t) = v^\e_0 e^{-t/\e} +\frac{1}{\e} \int_0^t u^\e(s) e^{-(t-s)/\e} ds +\frac{1}{\sqrt{\e}} \int_0^t  e^{-(t-s)/\e} dW(s).
\end{equation}

Moreover, if the initial conditions $u^\e_0$ are uniformly bounded in $\bold{H}$, then the solutions $u^\e$ satisfies the estimates:
\begin{equation}
\label{est1}
\sup_{\e > 0} \| u^\e \|_{L^\infty (\Omega;L^2(0,T;\bold{V}))} \leq C_T,
\end{equation}
\begin{equation}
\label{est2}
\sup_{\e > 0} \| u^\e \|_{L^\infty (\Omega;C([0,T];\bold{H}))} \leq C_T,
\end{equation}
and
\begin{equation}
\label{est3}
\sup_{\e > 0} \left\| \dfrac{\partial u^\e}{\partial t} \right\|_{L^\infty (\Omega;L^2(0,T;\bold{V}'))} \leq C_T.
\end{equation}
Also, if the initial conditions $v^\e_0$ are uniformly bounded in $L^2(\Omega; L^2(D)^3)$ we also have the estimate for $v^\e$:
\begin{equation}
\label{est4}
\sup_{\e > 0} \E \sup_{t\in[0,T]}  \| v^\e(t)\|^2_{L^2 (D)^3} \leq C_T.
\end{equation}
\end{theorem}
\begin{proof}
We prove the existence of solutions through a Galerkin approximation procedure. We consider $(e_k)_{k>0}$ a sequence of linearly independent elements in $\bold{V}\cap L^\infty(D)^3$ such that $span\{e_k\ | \ k>0\}$ is dense in $\bold{V}$. We define the $n$-dimensional space $\bold{V}_n$ for every $n>0$ as $span\{e_k\ | \ 1\leq  k \leq n\}$ and we denote by $\Pi_n$ the projection operator from $\bold{V}$ onto $\bold{V}_n$. 

Let us denote by $w^\e(t)$ the following process
\begin{equation}\label{w}
w^\e(t)=e^{-t/\e}v_0^\e+\frac{\sqrt{Q}}{\sqrt{\e}} \int_0^t  e^{-(t-s)/\e} dW(s) \in L^2(\Omega;C([0,T];L^2(D)^3).
\end{equation}

Moreover, there exists a random constant $C(\omega)$ almost surely finite such that
\begin{equation*}
\sup_{t\in[0,T]}  \| w^\e(t)\|^2_{L^2 (D)^3} \leq C(\omega).
\end{equation*}

Now, in order to prove the existence of solutions, we will proceed using a path wise argument; we fix $\omega\in\Omega$ and define  the Galerkin approximation 
$$(u^\e_n(\omega),z^\e_n(\omega) )\in C([0,T];\bold{H}_n)\cap L^2(0,T;\bold{V}_n))\times C([0,T];\Pi_n L^2(D)^3)$$ 
solution of the following system

\begin{equation}
\label{weaksolen}
\int_D\frac{\partial u^\e_n} {\partial t} (t) \phi dx + \int_D \nabla u^\e (t) \nabla \phi dx + \int_D \alpha^\e(z^\e_n(t)+w^\e(t)) u_n^\e(t) \phi dx =  \int_D f(t) \phi dx,
\end{equation}
for every $\phi \in \bold{V}_n$, $u^\e_n(0) =\Pi_n u^\e_0$,
\begin{equation}
\label{mildsolen}
\frac{\partial z^\e_n}{\partial t}(t) = -\frac{1}{\e}(z^\e_n(t)-u^\e_n(t)),\quad z^\e_n(0)=0,
\end{equation}
where 
\begin{equation}
\label{vne}
z^\e_n(t)=v^\e_n (t)- w^\e(t).
\end{equation}

Then,  we pass to the limit on $(u^e_n,z^\e_n )$ when $n \to \infty$. 

We write $u_n^\e(\omega,t,x) = \sum_{k=1}^n a^\e_k(\omega,t) e_k(x)$ and 
$z_n^\e(\omega,t,x) = \sum_{k=1}^n b^\e_k(\omega,t) e_k(x)$, and get the following system for the coefficients $a^\e_k$ and $b^\e_k$:
\begin{equation}
\label{weaksolen'}
\left\{
\begin{array}{rll}
& \displaystyle\sum_{k=1}^n\frac{\partial a^\e_k}{\partial t}(\omega,t) \displaystyle\int_D e_k e_j dx+\sum_{k=1}^n a^\e_k(\omega,t) \int_D \nabla e_k \nabla e_j dx &+\\
&\displaystyle \sum_{k=1}^n\int_D a^\e_k(\omega,t) \alpha^\e\left(w(\omega,t)+\sum_{k=1}^n b_k^\e(\omega,t) e_k\right) e_k e_jdx&= \displaystyle \int_D f(t) e_j dx,\\
\\
& \displaystyle\frac{\partial b^\e_k}{\partial t}(\omega,t)&=-\dfrac{1}{\e}\left( b^\e_k - a_k^\e \right),\ 1\leq k \leq n\\
\\
&a_k^\e(\omega,0)&=\displaystyle\int_D u^\e_0 e_k dx,\ 1\leq k \leq n\\
\\
&b_k^\e(\omega,0)&=0,\ 1\leq k \leq n
\end{array}
\right.
\end{equation}
for each $1\leq j \leq n$.
We make the following notations: $$a_{ij}=\displaystyle\int_D e_i(x) e_j(x) dx, \ b_{ij}=\displaystyle\int_D \nabla e_i(x) \nabla e_j(x) dx, \ f_j(s) = \int_D f(s,x) e_j(x) dx,$$
and
$$(F^\e_n)_{ij}(\omega,t,b_1,...b_n)= \int_D \alpha^\e\left(w(\omega,t)+\sum_{k=1}^n b_k e_k\right) e_i e_jdx$$
and the system is written with these notations as:
\begin{equation}
\label{weaksolen''}
\left\{
\begin{array}{rll}
& \displaystyle\sum_{k=1}^n\frac{\partial a^\e_k}{\partial t} a_{kj}+\sum_{k=1}^n a^\e_k b_{kj}+ \sum_{k=1}^n a^\e_k (F^\e_n)_{kj}(\omega,t,b_1^\e,...,b_n^\e)&= f_j(t),\\
\\
& \displaystyle\frac{\partial b^\e_k}{\partial t}&=-\dfrac{1}{\e}\left( b^\e_k - a_k^\e \right),\ 1\leq k \leq n\\
\\
&a_k^\e(0)&=\displaystyle\int_D u^\e_0 e_k dx,\ 1\leq k \leq n\\
\\
&b_k^\e(0)&=0,\ 1\leq k \leq n
\end{array}
\right.
\end{equation}
for each $1\leq j \leq n$.
Given the linearly independence of the sequence $(e_{k})_{k>0}$, the definition of the functions $(F^\e_n)_{ij}$ and the Lipschitz condition satisfied by $\alpha$, the system has unique solution $(a^\e_k)_{1\leq k \leq n}$,  
$(b^\e_k)_{1\leq k \leq n}$ in $C^1[0,T]$ for every $T>0$. This means that $u_n^\e$ and $z_n^\e=v_n^\e-w^\e$ is a solution for:
\begin{equation}
\label{weaksolen'''}
\left\{
\begin{array}{rll}
&\displaystyle\int_D\frac{\partial u^\e_n} {\partial t} (t) \phi dx + \int_D \nabla u_n^\e (t) \nabla \phi dx + \int_D \alpha^\e(z^\e_n(t)+w^\e(t)) u_n^\e(t) \phi dx &= \displaystyle\int_D f(t) \phi dx,\\
&d z_n^\e &=-\dfrac{1}{\e}\left( z^\e_n - u_n^\e \right),\\
\\
&u_n^\e(0)&=\Pi_n  u^\e_0,\\
\\
&z_n^\e(0)&=0,
\end{array}
\right.
\end{equation}

for every $\phi\in \bold{V}_n$. We take  $\phi= u_n^\e$ in \eqref{weaksolen'''}, use the positivity of $\alpha$ to derive that
\begin{equation}\nonumber
\begin{split}
\frac{\partial}{\partial t} \|u^\e_n \|^2_{L^2(D)^3} &\leq \|f(t)\|^2_{L^2(D)^3} +  \|u^\e_n \|^2_{L^2(D)^3}\Rightarrow\\
 \|u^\e_n \|^2_{L^2(D)^3} &\leq e^t \left( \|f\|_{L^2(0,T;L^2(D)^3)} + \|u^\e_0\|_{L^2(D)^3}\right),
\end{split}
\end{equation}
so
\begin{equation}
\label{estune1}
\sup_{n>0}\| u_n^\e\|_{L^\infty(0,T;L^2(D)^3)} \leq C_T(1+\|u^\e_0\|_\bold{H}).
\end{equation}
We also obtain that
\begin{equation}\nonumber
\begin{split}
\int_0^T  \|\nabla u^\e_n \|^2_{L^2(D)^{n\times n}} ds +\frac{1}{2} \| u_n^\e(T)\|^2_{L^2(D)^3}&\leq \int_0^T \int_D f(t) u_n^\e dxdt +  \frac{1}{2}\|u^\e_0 \|^2_{L^2(D)^3}\Rightarrow\\
\int_0^T  \|\nabla u^\e_n \|^2_{L^2(D)^{n\times n}} ds&\leq T \|f\|_{L^2(0,T;L^2(D)^3)}\| u_n^\e\|_{L^2(D)^3} +\frac{1}{2} \|u^\e_0\|_{L^2(D)^3},
\end{split}
\end{equation}
so 
\begin{equation}
\label{estune2}
\sup_{n>0}\| u_n^\e\|_{L^\infty(0,T;\bold{V})} \leq C_T(1+\|u^\e_0\|_{\bold{H}}).
\end{equation}
The estimates \eqref{estune1} and \eqref{estune2} imply using the first equation of the system \eqref{weaksolen'''} that
\begin{equation}
\label{estune3}
\sup_{n>0}\left\|\dfrac{\partial u_n^\e}{\partial t}\right\|_{L^\infty(0,T;\bold{V}_n')} \leq C_T(1+\|u^\e_0\|_{\bold{H}}).
\end{equation}

This means that  the sequence $u_n^\e$ is bounded in $L^2(0,T;\bold{V})\cap W^{1,2}(0,T;\bold{V}')$ which is compactly embedded in $L^2(0,T;\bold{H})$ (Theorem 2.1, page 271 from \cite{temam}).  Hence, there exists a 
subsequence $u_{n'}^\e(\omega)$ that converges strongly in $L^2(0,T;\bold{H})$ to some $u^\e(\omega)$ which is also a weak limit in $L^2(0,T;\bold{V})$ and a weak$^*$ limit in $L^\infty(0,T;\bold{H})$ and 
$u^\e \in L^2(0,T;\bold{V})\cap L^\infty (0,T;\bold{H})$.

We also have from \eqref{weaksolen'''} that
$$z^\e_{n'}(\omega,t)=\displaystyle\frac{1}{\e}\int_0^t e^{-(t-s)/\e} u^\e_{n'}(s)$$

will converge to $z^\e(\omega,t)=\displaystyle\frac{1}{\e}\int_0^t e^{-(t-s)/\e} u^\e(s)$ in $C([0,T];\bold{H})$.

We now pass to the limit when $n'\to\infty$ in the system \eqref{weaksolen'''} pointwise in $\omega\in\Omega$. We integrate the first equation over $[0,T]$, use the convergences of the sequences $u^\e_{n'}$ and $\dfrac{\partial u^\e_{n'}}{\partial t}$, so we get  that for every $\phi \in \bold{V}_n$
$$\lim_{n'\to\infty} \int_0^t \int_D\frac{\partial u^\e_{n'}} {\partial t} (s,x) \phi(x) dx ds=  \int_0^t\int_D\frac{\partial u^\e} {\partial t}(s,x) \phi (x)dxds$$
and
$$\lim_{n'\to\infty}   \int_0^t\int_D \nabla u_{n'}^\e (s) \nabla \phi dx ds=  \int_0^t\int_D \nabla u^\e (s) \nabla \phi dx ds .$$

Also
\begin{equation}
\begin{split}
\left|\int_0^t \int_D \alpha^\e(z^\e_{n'}+w^\e) u_{n'}^\e \phi  dxds- 
\int_0^t\int_D \alpha^\e(z^\e+w^\e) u^\e \phi  dx \right| &\leq \\
\left| \int_0^t\int_D \alpha^\e(z^\e_{n'}+w^\e) (u_{n'}^\e-u^\e) \phi  dxds \right|
+ \left|\int_0^t\int_D \left(\alpha^\e(z^\e_{n'}+w^\e) -\alpha^\e(z^\e+w^\e)\right) u^\e \phi dxds\right| &\leq \\
C \int_0^t\int_D (u_{n'}^\e-u^\e)^2 dxds+ C \int_0^t\int_D \left|z^\e_{n'} -z^\e\right| |u^\e| |\phi | dxds& \leq \\
C \| u_{n'}^\e-u^\e\|^2_{L^2(0,T;\bold{H})} + C \int_0^T \| z^\e_{n'} -z^\e \|_\bold{H} \| u^\e \|_\bold{H} \| \phi \|_{L^\infty(D)^3}.
\end{split}
\end{equation}
so we obtain that
$$\lim_{n'\to\infty}   \int_0^t\int_D \alpha^\e(z^\e_{n'}+w^\e) u_{n'}^\e \phi dxds =  
\int_0^t\int_D \alpha^\e(z^\e+w^\e) u^\e \phi dx ds .$$ 
We use the convergence for $z_{n'}^\e$  and obtain in the limit:
\begin{equation}
\label{weaksolen''''}
\left\{
\begin{array}{rll}
&\displaystyle\int_0^t \int_D\frac{\partial u^\e} {\partial t}(t) \phi  dxds + \int_0^t \int_D \nabla u^\e (t) \nabla \phi  dxds + \int_0^t \int_D \alpha^\e(z^\e+w^\e) u^\e \phi  dxds &= \displaystyle \int_0^t \displaystyle\int_D f(t) \phi  dxds,\\
&d z^\e &=-\dfrac{1}{\e}\left( z^\e - u^\e \right),\\
\\
&u^\e(0)&=u^\e_0,\\
\\
&z^\e(0)&=0,
\end{array}
\right.
\end{equation}
for every $\phi\in \bold{V}_n$, so by density it is true for any $\phi \in \bold{V}$.  Now, let  
$v^\e:=z^\e+w^\e$, then we deduce that $(u^\e, v^\e)$ is a solution for our initial system in the sense given by \eqref{weaksole} and \eqref{mildsole}. The solution   $(u^\e, v^\e)$ is measurable as the limit of the Galerkin approximation  $(u_n^\e, v_n^\e)$ which is measurable by construction.
Furthermore, given the uniform estimates for $u^\e_0$ it is easy to  obtain from \eqref{estune1}--\eqref{estune3} the estimates \eqref{est1}--\eqref{est3}.

Now, we prove the uniqueness. Let us assume the we have two solutions $\{u^\e_1, v^\e_1\}$ and $\{u^\e_2,v^\e_2\}$ for the system. Then,
\begin{equation}\nonumber
\begin{split}
\int_D (u^\e_2(t) -u^\e_1(t)) \phi dx  + \int_0^t \int_D (\nabla u^\e_2 (s) -\nabla u^\e_1(s)) \nabla \phi dx ds =
\int_0^t \int_D (\alpha^\e(v^\e_1) u^\e_1 - \alpha^\e(v^\e_2) u^\e_2)\phi dx ds,
\end{split}
\end{equation}
and
\begin{equation}
\nonumber
v^\e_2(t)-v^\e_1(t) = \frac{1}{\e} \int_0^t (u^\e_2(s)-u^\e_1(s)) e^{-(t-s)/\e} ds.
\end{equation}
we take $\phi = u^\e_2 - u^\e_1$ and we get:
\begin{equation}
\nonumber
\begin{split}
\int_D (u^\e_2(t) -u^\e_1(t))^2 dx  + \int_0^t \int_D (\nabla u^\e_2 (s) -\nabla u^\e_1(s))^2 dx ds =\\
-\int_0^t \int_D \alpha^\e(v^\e_2) (u^\e_2 - u^\e_1)^2 + \int_0^t \int_D (\alpha^\e(v^\e_1)-\alpha^\e(v^\e_2) )u^\e_1 (u^\e_2 - u^\e_1)\leq \\
C \int_0^t \int_D |v^\e_2-v^\e_1| |u^\e_1| u^\e_2-u^\e_1| dx ds \leq \\
C \int_0^t \|v^\e_2 - v^\e_1 \|_{L^2(D)^3} \|u^\e_1 \|_{L^4(D)^3}\|u^\e_2 - u^\e_1 \|_{L^4(D)^3} \leq \\
C \left(\int_0^t \|v^\e_2 - v^\e_1 \|^2_{L^2(D)^3} \|u^\e_1 \|^2_{L^4(D)^3}ds \right)^{1/2} \left(\int_0^t \|u^\e_2 - u^\e_1 \|^2_{L^4(D)^3}\right)^{1/2}\leq \\
C \int_0^t \|v^\e_2 - v^\e_1 \|^2_{L^2(D)^3} \|\nabla u^\e_1 \|^2_{L^2(D)^{3\times 3}}ds +\dfrac{1}{2} \int_0^t \|\nabla u^\e_2 - \nabla u^\e_1 \|^2_{L^2(D)^{3\times 3}}
\end{split}
\end{equation}
where we used H\"{o}lder's inequality and the imbedding of $H^1(D)^3$ into $L^4(D)^3$. Now, 
$$ \|v^\e_2(t) - v^\e_1(t) \|_{L^2(D)^3} \leq \dfrac{1}{\e} \int_0^t \| u^\e_2(s)-u^\e_1(s)\|_{L^2(D)^3} e^{-(t-s)/\e} \leq T \sup_{s\in [0,t]} \| u^\e_2(s)-u^\e_1(s)\|_{L^2(D)^3},$$
so we obtain:
\begin{equation}
\nonumber
\begin{split}
\sup_{s\in[0,t]}\|u^\e_2(t) - u^\e_1(t)\|^2_{L^(D)^3}  \leq 
C T \int_0^t  \sup_{r\in [0,s]} \| u^\e_2(r)-u^\e_1(r)\|_{L^2(D)^3}  \|\nabla u^\e_1(s) \|^2_{L^2(D)^{3\times 3}}ds.
\end{split}
\end{equation}
We use Gr\"{o}nwall's lemma for the function $\sup_{s\in[0,t]}\|u^\e_2(t) - u^\e_1(t)\|^2_{L^(D)^3}$ to obtain that:
$$\sup_{s\in[0,t]}\|u^\e_2(t) - u^\e_1(t)\|^2_{L^(D)^3} \leq \|u^\e_2(0) - u^\e_1(0)\|^2_{L^(D)^3} e^{C T \displaystyle \int_0^t \|\nabla u^\e_1(s) \|^2_{L^2(D)^{3\times 3}}ds},$$
which gives the uniqueness and this completes the proof.
\end{proof}

\begin{theorem}
\label{threg}
Assume that the initial conditions $u_0^\e$ are uniformly bounded in $\bold{V}$. Then the solution $u^\e$ will satisfy the improved estimates:
\begin{equation}
\label{est1'}
\sup_{\e > 0} \| u^\e \|_{L^\infty (\Omega;L^2(0,T;H^2(D)^3))} \leq C_T,
\end{equation}
\begin{equation}
\label{est2'}
\sup_{\e > 0} \| u^\e \|_{L^\infty (\Omega;C([0,T];\bold{V}))} \leq C_T,
\end{equation}
and
\begin{equation}
\label{est3'}
\sup_{\e > 0} \left\|\dfrac{ \partial  u^\e}{\partial t} \right\|_{L^\infty (\Omega;L^2(0,T; \bold{H})} \leq C_T.
\end{equation}
\begin{proof}
To show these estimates we go back to the Galerkin approximation used to show the existence. In the system \eqref{weaksolen'''} we take $\phi=\dfrac{\partial u_{n}^\e}{\partial t} (t)$ and get
\begin{equation}\nonumber
\displaystyle\int_D\left| \frac{\partial u^\e_n} {\partial t} (t)\right|^2 dx + \int_D \nabla u_n^\e (t) \nabla\frac{\partial u^\e_n} {\partial t} (t) dx \leq C \left\| \frac{\partial u^\e_n} {\partial t} (t)\right\|_{L^2(D)^3}\left(\|f(t)\|_{L^2(D)^3} + \|u_n^\e (t)\|_{L^2(D)^3} \right).
\end{equation}
We integrate on $[0,t]$ and use the estimates already obtained for $u^\e_n$ to get:
\begin{equation}\nonumber
\int_0^t \left\| \frac{\partial u^\e_n} {\partial t} (s)\right\|^2_{\bold{H}}ds +  \left\| \nabla u^\e_n(t)\right\|^2_{\bold{H}} \leq \left\| \nabla u^\e_n(0)\right\|^2_{\bold{H}} + C\int_0^t \left\| \frac{\partial u^\e_n} {\partial t} (s)\right\|_{\bold{H}},
\end{equation}
and from here
$$\sup_{\e>0} \sup_{n>0} \int_0^t\left\| \frac{\partial u^\e_n} {\partial t} (s)\right\|^2_{\bold{H}}ds \leq C_T,$$
and
$$\sup_{\e>0} \sup_{n>0} \sup_{t\in[0,T]}\left\| \nabla u^\e_n(t)\right\|^2_{\bold{H}} \leq C_T,$$
which will give us by passing to the limit on the subsequence $u^\e_{n'}$ \eqref{est3'} and 
\begin{equation}
\nonumber
\sup_{\e > 0} \| u^\e \|_{L^\infty (\Omega;L^\infty(0,T;\bold{V}))} \leq C_T,
\end{equation}
We use now the first equation from \eqref{system} and the regularity theorem for the stationary Stokes equation from \cite{temam} to obtain \eqref{est1'}. We get \eqref{est2'} by using Lemma 1.2, section 1.4 from \cite{temam}. 
\end{proof}
\end{theorem}

\section{The fast motion equation}
\label{sec3}

In this section, we present some facts for the invariant measure associated
with (\ref{v}) and introduce an auxiliary eigenvalue problem.
We consider the following problem for fixed $\xi \in L^2(D)^3$:
\begin{equation}
\label{v}
\left\{
\begin{array}{ll}
dv^\xi  &= - (v^\xi-\xi)dt + \sqrt{Q} dW,\\
v(0) &= \eta.
\end{array}
\right.
\end{equation}
This equation admits a unique mild solution $v^\xi(t)\in L^2(\Omega; C(0,T;L^2(D)^3))$ given by:
\begin{equation}
\label{vxieta}
v^\xi(t) = \eta e^{-t} +\xi(1-e^{-t}) + \int_0^t e^{-(t-s)}\sqrt{Q} dW .
\end{equation}
When needed to specify the dependence with respect to the initial condition the solution will be denoted by $v^{\xi,\eta}(t)$. The following estimate can be 
derived for $v^{\xi,\eta}(t)$.
\begin{lemma}
\label{l}
\begin{equation}
\E \| v^{\xi,\eta}(t) \|^2_{L^2(D)^3} \leq \|\eta\|^2_{L^2(D)^3} e^{-t} + \|\xi\|^2_{L^2(D)^3} + Tr Q.
\end{equation}
\end{lemma}
\begin{proof} It's enough to use the It\^o formula for $\| v^{\xi,\eta}(t) \|^2_{L^2(D)^3}$. 
\end{proof}
\subsection{The asymptotic behavior of the fast motion equation}
\label{subs21}
Let us define the transition semigroup $P_t^\xi$ associated to the equation \eqref{v}
\begin{equation}
P_t^\xi \Psi (\eta) = \E \Psi(v^{\xi,\eta}(t)),
\end{equation}
for every $\Psi \in B_b(L^2(D)^3)$ and every $\eta \in L^2(D)^3$.
It is easy to verify that $P_t^\xi$ is a Feller semigroup because $\mathbb{P}$ a.s.
\begin{equation}\label{feller}
\|v^{\xi,\eta_1} - v^{\xi,\eta_2}\|^2_{L^2(D)^3} \leq e^{-2t} \|\eta_1-\eta_2 \|_{L^2(D)^3}.
\end{equation}
We also denote by $\mu^\xi$ the associated invariant measure on $L^2(D)^3$. We recall that it is invariant for the semigroup $P_t^\xi$ if
$$\int_{L^2(D)^2} P_t^\xi \Psi (z) d\mu^\xi(z) = \int_{L^2(D)^2}  \Psi (z) d\mu^\xi(z),$$
for every $\Psi \in B_b(L^2(D)^3)$.
It is obvious that $v^\xi$ is a stationary gaussian process. The equation \eqref{v} admits a unique ergodic invariant measure $\mu^\xi$ that is strongly mixing and gaussian with mean $\xi$ and covariance operator $Q$. All these results can be found in \cite{DPZ2} or \cite{Cerrai}. 

We denote by $\mathcal{L}^\xi$ the Kolmogorov operator associated to the semigroup $P_t^\xi$, which is given by
\begin{equation}
\label{Kop}
\mathcal{L}^\xi \Psi(\eta) = \dfrac{\partial \Psi} {\partial \eta} (\eta) \cdot (-\eta+\xi) + \dfrac{1}{2} Tr\left[Q \dfrac{\partial ^2\Psi}{\partial \eta^2}\right],
\end{equation}
for every $\Psi \in C^2_b(L^2(D)^3)$.

For $\lambda>0$, let us introduce the eigenvalue problem associated to the operator $\mathcal{L}^\xi$:
\begin{equation}
\label{eigen}
\lambda \Psi(\eta) - \mathcal{L}^\xi \Psi (\eta) = F(\eta),
\end{equation}
for $\eta \in L^2(D)^3$ and $F\in C^1_b(L^2(D)^3)$. The equation \eqref{eigen} has a unique strict solution (see Section 9 from \cite{DPZ}). Moreover the solution is given by (see Theorem 3.8 in \cite{CF2008}):
\begin{equation*}
\Psi (\eta,\xi) = \int_0^\infty e^{-\lambda t} P_t^\xi F (\eta) dt.
\end{equation*}

\section{Passage to the limit}
\label{sec4}
The main goal of this section is to pass to the limit in the system \eqref{system} when $\e\to 0$.

We introduce the following averaged operators:
\begin{equation}
\label{defoalphae}
\o{\alpha^\e}: L^2(D)^3 \to L^\infty(D),\ \ \o{\alpha^\e}(\xi) = \int_{L^2(D)^3}\alpha^\e(\eta) d\mu^\xi (\eta)
\end{equation}

\begin{equation}
\label{defoalpha}
\o{\alpha} (\xi) = \int_{L^2(D)^3} \left(\int_Y \alpha(y,z)dy\right) d\mu^\xi(z).
\end{equation}

We remark that $\alpha^\e$ as an operator from $L^2(D)^3$ to $L^2(D)^3$ is Lipschitz and $L^2(D)^3$ is separable, so Pettis Theorem implies that $\alpha^\e: L^2(D)^3 \to L^2(D)^3$ is measurable. The boundedness of $\alpha^\e$ implies the integrability with respect to the probability measure $\mu^\xi$, so $\alpha^\e$ is well defined (see Chapter 5, Sections 4 and 5 from \cite{Yosida} for details). The same considerations hold also for the operators $z\in L^2(D)^3 \to \displaystyle\int_Y \alpha(y,z)dy \in L^\infty(D)^3$, so $\o{\alpha}$ is also well defined.

Our main result is given by the next theorem.
\begin{theorem}
\label{thconv}
Assume the sequence $u^\e_0$ is uniformly bounded in $\bold{V}$ and strongly convergent in $\bold{H}$ to some function $u_0$, and $v^\e_0$ is uniformly bounded in $L^2(\Omega;L^2(D)^3)$. Then, there exists $\o{u} \in L^2(0,T;\bold{V})$ such that $u^\e$ converges in probability to $\o{u}$ in $L^2(0,T;\bold{V})$ and $\o{u}$ is the solution of the following deterministic equation:
\begin{equation}
\label{eqou}
\left\{
\begin{array}{rll}
\dfrac{\partial \o{u}}{\partial t} &= \Delta \o{u} - \nabla \o{p} - \o{\alpha} (\o{u}) \o{u} + f &\mbox{ in }\ D , \\
\operatorname{div}\o{u} &=0 &\mbox{ in }\  D, \\
\o{u} &=0 &\mbox{ on }\  \partial D, \\
\o{u}(0) & = u_0&\mbox{ in }\  D.
\end{array}
\right.
\end{equation}
\end{theorem}

Let us explain the main ideas involved in the proof of this convergence. The uniform bounds for $u^\e$ provided by Theorem \ref{threg} imply that the sequence is tight in $L^2(0,T;H^1_0(D)^3)$, so there exists a limit $\o{u}$ in distribution. We apply after that Skorokhod theorem to get another sequence $\wt{u^\e}$ defined on some probability space $\Omega'$, with same distribution as $u^\e$ that converges for a.e. 
$\omega\in \Omega'$ to some $\wt{\o{u}}$ in $L^2(0,T; H^1(D)^3)$. We show that $\wt{\o{u}}$ is deterministic and get an equation for it by passing to the limit in expected value in the variational formulation. More precisely, we prove first that:
\begin{equation}
\label{dif}
\begin{split}
\lim_{\e \to 0} &\  \E \left | \int_0^T \int_D \left( \alpha^\e(v^\e(t)) u^\e(t) - \o{\alpha}(\o{u}(t)) \o{u}(t) \right)\phi \psi(t) dx dt \right | = 0,
\end{split}
\end{equation}
for every $\phi \in \bold{V}$ and any $\psi \in C[0,T]$ with $\psi(T)=0$.
We rewrite it as:
$$\int_0^T \int_D \left( \alpha^\e(v^\e(t)) u^\e(t) - \o{\alpha}(\o{u}(t)) \o{u}(t) \right)\phi \psi(t) dx dt =  S^\e_1 + S^\e_2 +S^\e_3, $$
where
\begin{equation}\nonumber
S^\e_1  = \int_0^T  \int_D \left(\alpha^\e(v^\e(t)) - \o{\alpha^\e}(u^\e(t))\right)u^\e(t)\phi \psi(t) dxdt,
\end{equation}
\begin{equation}\nonumber
S^\e_2  = \int_0^T \int_D \left( \o{\alpha^\e}(u^\e(t))u^\e(t) - \o{\alpha^{\e}}(\o{u}(t))\o{u}(t) \right)\phi \psi(t)dx dt,
\end{equation}
and
\begin{equation}\nonumber
S^\e_3  = \int_0^T  \int_D \left(\o{\alpha^{\e}}(\o{u}(t))\o{u}(t) - \o{\alpha}(\o{u}(t))\o{u}(t)\right)\phi \psi(t) dxdt.
\end{equation}
This convergence requires several preliminary steps. The first step is performed in Subsection \ref{subs42} where we get uniform estimates for $\Psi^\e(\eta,\xi)$, $\dfrac{\partial \Psi^\e}{\partial \eta}(\eta,\xi)$ and $\dfrac{\partial \Psi^\e}{\partial \xi}(\eta,\xi)$, where $\Psi^\e: L^2(D)^3 \times L^2(D)^3 \to \mathbb{R}$ solves 
\begin{equation}\label{eqpsie}
c(\e) \Psi^\e(\eta,\xi) - \mathcal{L}^\xi \Psi^\e(\eta,\xi) =  \int_D \left(\alpha^\e(\eta) - \o{\alpha^\e}(\xi)\right)\xi\phi dx.
\end{equation}
for $c(\e) >0$ to be chosen later. This is an eigenvalue problem of the type \eqref{eigen} so the solution is given by
\begin{equation}
\label{psie}
\Psi^\e (\eta,\xi) = \int_0^\infty e^{-c(\e) t} P_t^\xi \left[ \int_D \left(\alpha^\e(\cdot) - \o{\alpha^\e}(\xi)\right)\xi\phi dx (\eta)\right] dt.
\end{equation}
      
The second step is done in Subsection \ref{subs43} where we prove the convergence to $0$ for $S^\e_1$ by applying It\^{o}'s formula to $\Psi^\e(v^\e(t),u^\e(t))$ and getting an expression for $S^\e_1$ in terms of $\Psi^\e$ and its first order Fr\'{e}chet derivatives. The result is contained in Lemma \ref{lemmaconv}. In Subsection \ref{subs44} we do the third step, the convergence to $0$ of $S^\e_3$, which is showed in Lemma \ref{lemmaconv'}. 

The sequence $\wt{u^\e}$ given by Skorokhod theorem converges a.s.  to $\wt{\o{u}}$ strongly in $L^2(0,T;H^1_0(D)^3)$ so 
\begin{equation}\label{dif1}
\lim_{\e \to 0} \left | \int_0^T \int_D \left(\wt{u^\e} (t)-\wt{\o{u}} (t)\right) \phi \psi'(t) dx dt  - \int_0^T \int_D \left( \nabla \wt{u^\e} - \nabla\wt{\o{u}} \right) \nabla \phi \psi(t) dx dt\right | = 0, \quad a.s.  \end{equation}
The equations \eqref{dif} and \eqref{dif1} imply that $\wt{\o{u}}$ satisfies almost surely the variational formulation associated with \eqref{eqou}, so $\wt{\o{u}}$ and $\o{u}$ are deterministic and as a consequence the convergence of the sequence $u^\e$ to $\o{u}$ will be in probability.
Before proceeding with the proof of Theorem \ref{thexunou}, let us first study system  \eqref{eqou}.

\subsection{Well-possedness for the averaged equation \eqref{eqou}}
\label{subs41}
\begin{theorem}\label{thexunou}
Assume $f\in L^2(0,T;L^2(D)^3)$ and $\o{\alpha} \in C_b^1(\mathbb{R}^3)$ is a positive function. Then, for any $u_0 \in \bold{H}$ the system \eqref{eqou} admits a unique solution $\o{u} \in C([0,T];\bold{H})\cap L^2(0,T;\bold{V})$ with $\dfrac{\partial \o{u}}{\partial t} \in L^2(0,T;\bold{V}')$ in the following sense:
\begin{equation}
\label{weaksolou}
\int_D \o{u}(t) \phi dx - \int_D u_0 \phi dx + \int_0^t \int_D \nabla \o{u} (s) \nabla \phi dx ds +\int_0^t \int_D \o{\alpha}(\o{u}) \o{u} \phi dx ds = \int_0^t \int_D f(s) \phi dx ds,
\end{equation}
for every $t\in[0,T]$ and every $\phi \in \bold{V}$. Moreover, if the initial condition $u_0 \in \bold{V}$, then $\o{u}$ has the improved regularity, $\o{u} \in L^2(0,T;H^2(D)^3) \cap C([0,T];\bold{V})$ and $\dfrac{ \partial  \o{u}}{\partial t} \in L^2(0,T; \bold{H})$.
\begin{proof}
The proof of existence of solutions is similar to the proof  of system \eqref{system}, using a Galerkin approximation procedure. The finite dimensional approximation $\o{u}_n$ will exist as in Theorem \ref{thexun} and will solve
\begin{equation}
\label{weaksoloun}
\int_D\frac{\partial \o{u}_n} {\partial t} (t) \phi dx + \int_D \nabla \o{u}_n (t) \nabla \phi dx + \int_D \o{\alpha}(\o{u}_n) \o{u}_n \phi dx =  \int_D f(t) \phi dx,
\end{equation}
for every $\phi \in C([0,T], \bold{V}_n)$, and $\o{u}_n(0) =\Pi_n u_0$. We take $\phi = \o{u}_n(t)$, and get:
\begin{equation*}
\begin{split}
&\int_D\frac{\partial \o{u}_n} {\partial t} (t) \o{u}_n(t) dx + \int_D \|\nabla \o{u}_n (t)\|^2 dx \leq   \int_D f(t) u_n(t)dx \Rightarrow \\
&\dfrac{ \partial} {\partial t} \|\o{u}_n(t) \|^2_{L^2(D)^3} \leq \|f(t) \|^2_{L^2(D)^3} + \|\o{u}_n(t) \|^2_{L^2(D)^3} \Rightarrow \\
&\|\o{u}_n(t) \|^2_{L^2(D)^3} \leq c +\int_0^t \|\o{u}_n(s) \|^2_{L^2(D)^3} ds.
\end{split}
\end{equation*}
We use Gr\"{o}nwall's lemma and get:
\begin{equation}\label{estunou1}
\sup _{n > 0}\| \o{u}_n \|_{C([0,T]; \bold{H}} \leq C_T,
\end{equation}
and from here we also obtain
\begin{equation}\label{estunou2}
\sup _{n > 0}\|\nabla  \o{u}_n \|_{L^2(0,T;L^2(D)^{3 \times 3})} \leq C_T,
\end{equation}
and
\begin{equation}\label{estunou3}
\sup _{n > 0}\left\| \dfrac{ \partial \o{u}_n} {\partial t} \right\|_{L^2(0,T;\bold{V}')} \leq C_T.
\end{equation}
So there exists a subsequence $\o{u}_{n'}$ and a function $\o{u} \in L^\infty(0,T; {\bf H}) \cap L^2(0,T;\bold{V})$ such that $\o{u}_{n'}$ converges weakly star in $L^\infty(0,T; \bold{H})$ and weakly to $L^2(0,T;\bold{V})$ to $\o{u}$ and also $\dfrac{\partial \o{u}_{n'}}{\partial t} $ converges to $\dfrac{\partial \o{u}}{\partial t} $ weakly in $L^2(0,T;\bold{V}')$. We apply again now Theorem 2.1, page 271 and Lemma 1.2 page 260 from \cite{temam} to obtain that $\o{u}_{n'}$ converges strongly in $L^2(0,T;\bold{H})$ and in $C([0,T];\bold{H})$ to $\o{u}$. We then pass to the limit and obtain that $\o{u}$ is a weak solution for \eqref{eqou}.

Now, to show uniqueness we assume to have two solutions $\o{u}_1$ and $\o{u}_2$  in $ C([0,T];\bold{H})\cap L^2(0,T;\bold{V})$ and substract the variational formulations. We get:
\begin{equation}\nonumber
\begin{split}
\int_D (\o{u}_2(t) -\o{u}_1(t)) \phi dx  + \int_0^t \int_D (\nabla \o{u}_2 (s) -\nabla \o{u}_1(s)) \nabla \phi dx ds =
\int_0^t \int_D (\o{\alpha}(\o{u}_1) \o{u}_1 - \o{\alpha}(\o{u}_2) \o{u}_2)\phi dx ds.
\end{split}
\end{equation}
We take $\phi = \o{u}_2 - \o{u}_1$ and write
\begin{equation*}
\begin{split}
(\o{\alpha}(\o{u}_1) \o{u}_1 - \o{\alpha}(\o{u}_2) \o{u}_2) (\o{u}_2 - \o{u}_1) &= - \o{\alpha}(\o{u}_1) (\o{u}_2 - \o{u}_1)^2  + \o{u}_2 (\o{\alpha}(\o{u}_2) - \o{\alpha}(\o{u}_1))(\o{u}_2 - \o{u}_1)\\
&\leq C \o{u}_2 (\o{u}_2 - \o{u}_1)^2.
\end{split}
\end{equation*}
We get
\begin{equation*}
\begin{split}
\| \o{u}_2(t) -\o{u}_1(t) \|^2_{L^2(D)^3} + \int_0^t \| \nabla \o{u}_2(s) &- \nabla \o{u}_1(s)\|^2_{L^2(D)^{3 \times 3}} \leq C \int_0^t \int_D \o{u}_2(s) (\o{u}_2(s) - \o{u}_1(s))^2 dx ds\\
& \leq C\int_0^t \|\o{u}_2(s) \|_{L^2(D)^3} \| \o{u}_2(s) - \o{u}_1(s)\|^2_{L^4(D)^3} ds.
\end{split}
\end{equation*}
after using H\"{o}lder's inequality. Interpolation inequality and Young's inequality give 
\begin{equation*}
\begin{split}
\| \o{u}_2(s) - \o{u}_1(s)\|^2_{L^4(D)^3}& \leq \| \o{u}_2(s) - \o{u}_1(s)\|^{1/2}_{L^2(D)^3}\| \nabla \o{u}_2(s) - \nabla \o{u}_1(s)\|^{3/2}_{L^2(D)^{3\times 3}}\\
&\leq c(\e)  \| \o{u}_2(s) - \o{u}_1(s)\|^{2}_{L^2(D)^3} +\e \| \nabla \o{u}_2(s) - \nabla \o{u}_1(s)\|^{2}_{L^2(D)^{3\times 3}},
\end{split}
\end{equation*}
so we obtain for a convenient choice of $\e$
\begin{equation*}
\begin{split}
\| \o{u}_2(t) -\o{u}_1(t) \|^2_{L^2(D)^3}  \leq c(\e) \int_0^t \| \o{u}_2(s) -\o{u}_1(s) \|^2_{L^2(D)^3}.
\end{split}
\end{equation*}
We get uniqueness from here by applying Gr\"{o}nwall's lemma.

Let us now assume that the initial condition $u_0 \in \bold{V}$. We use the equation \eqref{weaksoloun} with $\phi = \dfrac{\partial \o{u}_n}{\partial t}$:
\begin{equation}
\label{weaksoloun'}
\int_D\left(\frac{\partial \o{u}_n} {\partial t} (t)\right)^2 dx + \int_D \nabla \o{u}_n (t) \nabla \dfrac{\partial \o{u}_n}{\partial t} (t) dx + \int_D \o{\alpha}(\o{u}_n(t)) \o{u}_n(t) \dfrac{\partial \o{u}_n}{\partial t} (t)dx =  \int_D f(t) \dfrac{\partial \o{u}_n}{\partial t} (t)dx,
\end{equation}
we integrate it over $[0,T]$, and use H\"{o}lder's inequality:
$$ \left\| \frac{\partial \o{u}_n} {\partial t}\right\|^2_{L^2([0,T],L^2(D)^3)} + \| \nabla \o{u}_n (T)\|^2_{L^2(D)^{3 \times 3}} - \| \nabla \o{u}_n (0)\|^2_{L^2(D)^{3 \times 3}} \leq C \left\| \frac{\partial \o{u}_n} {\partial t}\right\|_{L^2([0,T],L^2(D)^3)}, $$
which will imply that $\dfrac{\partial \o{u}_n} {\partial t} \in L^2([0,T], {\bf H})$ uniformly bounded. This gives from that $\Delta \o{u}_n \in L^2([0,T]; {\bf H} )$ which based on the regularity theorem from the stationary Stokes equation implies that $\o{u}_n \in L^2([0,T], H^2(D)^3\cap {\bf H})$ and informally bounded. We deduce from here that $\o{u} \in L^2([0,T]; {\bf H}) \cap L^2([0,T], H^2(D)^3)$ and based on Lemma 1.2 page 260 from \cite{temam} that $\o{u} \in C([0,T]; \bold{V})$.

\end{proof}
\end{theorem}

\subsection{Estimates for $\Psi^\e$}
\label{subs42}
We recall that according to \eqref{psie}
$$\Psi^\e (\eta,\xi) = \int_0^\infty e^{-c(\e) t} P_t^\xi \left[ \int_D \left(\alpha^\e(\cdot) - \o{\alpha^\e}(\xi)\right)\xi\phi dx \right] (\eta)dt,$$
defined for a fixed $\phi \in \bold{V} \cap L^\infty(D)^3$.
\begin{lemma}\label{psi}
\begin{equation}
\label{estpsi}
|\Psi^\e(\xi,\eta)|  \leq c \|\xi\|_{L^2} \|\phi\|_{L^\infty} \left(1+\|\eta\|_{L^2} +\|\xi\|_{L^2} \right).
\end{equation}
\end{lemma}
\begin{proof}
\begin{equation}
 |\Psi^\e(\xi,\eta)| \leq \int_0^t e^{-c(\e) t} \left|  P_t^\xi \left[\int_D \left(\alpha^\e(\cdot) - \o{\alpha^\e}(\xi)\right)\xi\phi dx\right] (\eta) \right|
\end{equation}
We denote by $F^\e(\eta,\xi) = \displaystyle\int_D \alpha^\e(\eta) \xi \phi dx$, which due to the Lipschitz condition for the function $\alpha$ is a Lipschitz function of $\eta$ with the constant
$c\|\xi\|_{L^2(D)^3} \|\phi\|_{L^\infty(D)^3}$ and Lipschitz with respect to $\xi$ with the constant $c \|\phi\|_{L^\infty(D)^3}$.

Using that the measure $\mu^\xi$ is invariant, we have that

\begin{equation}
\begin{split}
P_t^\xi F^\e(\eta,\xi) - \int_{L^2(D)^3} F^\e(w,\xi) d \mu^\xi(w) &= \int_{L^2(D)^3} \left[P_t^\xi F^\e(\eta,\xi) - P_t^\xi F^\e(w,\xi)\right] d\mu^\xi(w)\\
& = \int_{L^2(D)^3} \left[\E F^\e(v^{\xi,\eta}(t)) - \E F^\e(v^{\xi,w}(t))\right] d\mu^\xi(w)\\
& \leq c\|\xi\|_{L^2} \|\phi\|_{L^\infty} \int_{L^2(D)^3} \E \| (v^{\xi,\eta}(t)) - (v^{\xi,w}(t))\|_{L^2} d\mu^\xi(w)\\
 \intertext{ using \eqref{feller} }
& \leq c e^{-t}\|\xi\|_{L^2} \|\phi\|_{L^\infty} \int_{L^2(D)^3} \| (\eta-w)\|_{L^2} d\mu^\xi(w)\\
& \leq c e^{-t}\|\xi\|_{L^2} \|\phi\|_{L^\infty} \left(\|\eta\|_{L^2} +\int_{L^2(D)^3} \|w\|_{L^2}d\mu^\xi(w)  \right).
\end{split}
\end{equation}
Again using the invariance of $\mu^\xi$ and Lemma \ref{l}

\begin{equation}
\begin{split}
\int_{L^2(D)^3} \|w\|_{L^2}d\mu^\xi(w) &= \int_{L^2(D)^3} P_t^\xi \|w\|_{L^2}d\mu^\xi(w)\\
& = \int_{L^2(D)^3} \E    \|v^{\xi,w}(t)\|_{L^2}d\mu^\xi(w)\\
&\leq  \int_{L^2(D)^3} c(1 + \|\xi\|_{L^2}+ e^{-t}\|v^\xi(0,w)\|_{L^2})d\mu^\xi(w)\\
&\leq  c(1+\|\xi\|_{L^2})
\end{split}
\end{equation}
\end{proof}

Now, we estimate the Fr\^{e}chet partial derivatives of $\Psi^\e$ with respect to $\eta$ and $\xi$.
\begin{lemma}
\label{psi1}
\begin{equation}\label{estpsi1}
\left\| \dfrac{\partial \Psi^\e}{\partial \eta} (\eta,\xi) \right\|_{L^2(D)^3} \leq c\|\xi\|_{L^2(D)^3} \| \phi \|_{L^\infty(D)^3}.
\end{equation}
\end{lemma}
\begin{proof}
Let us show first that for every $t>0$,
\begin{equation}\label{psi1'}
\dfrac{\partial}{\partial \eta} P_t^\xi F^\e(\cdot,\xi) (\eta) = e^{-t} \E \left(\frac{ \partial \alpha^\e} {\partial \eta} \left(v^{\xi,\eta}(t))\right) \xi\cdot \phi\right),
\end{equation}
which means that
\begin{equation*}
\lim_{\|w\|_{L^2(D)^3} \to 0}  \dfrac {P_t^\xi F^\e(\cdot,\xi) (\eta+ w)  - P_t^\xi F^\e(\cdot,\xi) (\eta) -\displaystyle\int_De^{-t} \E \left( \frac{ \partial \alpha^\e} {\partial \eta} \left(v^{\xi,\eta}(t))\right) \xi\cdot \phi\right) w dx}{\|w\|_{L^2(D)^3}}.
\end{equation*}
We have:
\begin{equation*}
\begin{split}
&P_t^\xi F^\e(\cdot,\xi) (\eta+ w)  - P_t^\xi F^\e(\cdot,\xi) (\eta) -\displaystyle\int_De^{-t} \E \left( \frac{ \partial \alpha^\e} {\partial \eta} \left(v^{\xi,\eta}(t))\right) \xi\cdot \phi\right) w dx\\
=& \E \left( F^\e(v^{\xi,\eta + w}(t),\xi)-F^\e(v^{\xi,\eta}(t),\xi)\right)  -\displaystyle\int_De^{-t} \E \left( \frac{ \partial \alpha^\e} {\partial \eta} \left(v^{\xi,\eta}(t))\right) \xi\cdot \phi\right) w dx\\
=&\E \left( F^\e(v^{\xi,\eta + w}(t),\xi)-F^\e(v^{\xi,\eta}(t),\xi) - e^{-t}\displaystyle \int_D \dfrac{ \partial \alpha^\e} {\partial \eta} \left(v^{\xi,\eta}(t)\right)\cdot w \xi\cdot \phi dx\right)\\
=&\E \left(\displaystyle \int_D \left(\alpha^\e \left( v^{\xi,\eta + w}(t)\right) - \alpha^\e\left( v^{\xi,\eta}(t) \right)-\dfrac{ \partial \alpha^\e} {\partial \eta} \left(v^{\xi,\eta}(t)\right)\cdot e^{-t}w \right) \xi \cdot \phi dx \right).
\end{split}
\end{equation*}
But from \eqref{vxieta}, $v^{\xi,\eta + w}(t) = v^{\xi,\eta}(t) +e^{-t} w$, so
\begin{equation*}
\begin{split}
&P_t^\xi F^\e(\cdot,\xi) (\eta+ w)  - P_t^\xi F^\e(\cdot,\xi) (\eta) -e^{-t}\E\left(\displaystyle \int_D \dfrac{ \partial \alpha^\e} {\partial \eta} \left(v^{\xi,\eta}(t)\right)\cdot w \xi\cdot \phi dx\right)\\
=&\E \left(\displaystyle \int_D \left(\alpha^\e \left( v^{\xi,\eta}(t) + e^{-t} w\right) - \alpha^\e\left( v^{\xi,\eta}(t) \right)-\dfrac{ \partial \alpha^\e} {\partial \eta} \left(v^{\xi,\eta}(t) \right)\cdot e^{-t}w  \right) \xi \cdot \phi dx \right)
\end{split}
\end{equation*}
Assume that \eqref{psi1'} does not hold, so there exists a sequence $w_n\in L^2(D)^3$ with the norms converging to $0$, such that
\begin{equation*}\begin{split}
\lim_{n\to \infty} & \E  \displaystyle \int_D \dfrac{\left(\alpha^\e \left( v^{\xi,\eta}(t) + e^{-t} w_n\right) - \alpha^\e\left( v^{\xi,\eta}(t) \right)-\dfrac{ \partial \alpha^\e} {\partial \eta} \left(v^{\xi,\eta}(t) \right)\cdot e^{-t}w_n  \right) \xi\cdot \phi dx }{\| w_n\|_{L^2(D)^3}}  \neq 0.
\end{split}
\end{equation*}
We can assume w.r.g. that $w_n \to 0 $ pointwise in $D$. This implies, by the differentiability of the function $\alpha$ that 
$$\dfrac{\left(\alpha^\e \left(v^{\xi,\eta}(t) + e^{-t} w_n\right) - \alpha^\e\left(v^{\xi,\eta}(t) \right)-\dfrac{ \partial \alpha^\e} {\partial \eta} \left(v^{\xi,\eta}(t) \right)\cdot e^{-t}w_n  \right) \xi \cdot \phi}{\| w_n\|_{L^2(D)^3}}$$
converges to $0$ pointwise in $D$. We get a contradiction after applying  the dominated convergence theorem.

From \eqref{psi1'} we obtain by differentiating inside the integral that 
\begin{equation}\label{psi1''}
\dfrac{\partial \Psi^\e}{\partial \eta} (\eta,\xi) = \int_0^\infty e^{-c(\e)t-t} \E \left(\frac{ \partial \alpha^\e} {\partial \eta} \left(v^{\xi,\eta}(t))\right) \xi\cdot \phi\right) dt
\end{equation}
which implies \eqref{estpsi1}.

\end{proof}
\begin{lemma}
\label{psi2}
\begin{equation}\label{estpsi2}
\left\| \dfrac{\partial \Psi^\e}{\partial \xi} (\eta,\xi) \right\|_{L^2(D)^3} \leq  \dfrac{c}{c(\e)}\left(1+ \|\xi\|_{L^2(D)^3}\right) \| \phi \|_{L^\infty(D)^3}.
\end{equation}
\end{lemma}
\begin{proof}
Similarly as in the previous lemma we show that
\begin{equation}\label{psi2'}
\dfrac{\partial}{\partial \xi} P_t^\xi F^\e(\cdot,\xi) (\eta) =\E \left(\left(1-  e^{-t}\right) \frac{ \partial \alpha^\e} {\partial \eta} \left(v^{\xi,\eta}(t))\right) \xi\cdot \phi+ \alpha^\e \left(v^{\xi,\eta}(t)\right) \phi\right)
\end{equation}
and
\begin{equation}\label{psi2''}
\dfrac{\partial}{\partial \xi}  \int_{L^2(D)^3} F^\e(w,\xi) d \mu^\xi(w) = \int_{L^2(D)^3}\alpha^\e (w) \phi d \mu^\xi(w)  + \int_{L^2(D)^3} \dfrac{\partial \alpha^\e}{\partial \eta}\left(w \right) \phi \xi d \mu^\xi(w).
\end{equation}
We write:
\begin{equation*}\begin{split}
P_t^{\xi + z} &F^\e (\cdot, \xi +  z) - P_t^{\xi} F^\e (\cdot, \xi) = P_t^{\xi + z} F^\e (\cdot, \xi +  z) - P_t^{\xi} F^\e (\cdot, \xi + z) +P_t^{\xi} F^\e (\cdot, \xi +  z) - P_t^{\xi} F^\e (\cdot, \xi)\\
&= \E \left( F^\e(v^{\xi+z, \eta}(t), \xi+z) - F^\e(v^{\xi, \eta}(t), \xi+z)+ F^\e(v^{\xi, \eta}(t), \xi+z) - F^\e(v^{\xi, \eta}(t), \xi)    \right)\\
\end{split}\end{equation*}

We have:
\begin{equation*}
\begin{split}
&P_t^{\xi+z} F^\e(\cdot,\xi+z) - P_t^\xi F^\e(\cdot,\xi+z) -\displaystyle\int_D (1-e^{-t}) \E \left( \frac{ \partial \alpha^\e} {\partial \xi} \left(v^{\xi,\eta}(t))\right) \xi\cdot \phi\right) z dx\\
=& \E \left( F^\e(v^{\xi+z,\eta}(t),\xi+z)-F^\e(v^{\xi,\eta}(t),\xi+z)\right)  -\displaystyle\int_D(1-e^{-t} \E \left( \frac{ \partial \alpha^\e} {\partial \xi} \left(v^{\xi,\eta}(t))\right) \xi\cdot \phi\right) z dx\\
=&\E \left( F^\e(v^{\xi+z,\eta}(t),\xi+z)-F^\e(v^{\xi,\eta}(t),\xi+z) - (1-e^{-t})\displaystyle \int_D 
\dfrac{ \partial \alpha^\e} {\partial \xi} \left(v^{\xi,\eta}(t)\right)\cdot z \xi\cdot \phi dx\right)\\
=&\E \left(\displaystyle \int_D \left(\alpha^\e \left( v^{\xi+z,\eta}(t)\right) - \alpha^\e\left( v^{\xi,\eta}(t) \right)-\dfrac{ \partial \alpha^\e} {\partial \xi} \left(v^{\xi,\eta}(t)\right)\cdot (1-e^{-t})z \right) \xi \cdot \phi dx \right).
\end{split}
\end{equation*}
But from \eqref{vxieta}, $v^{\xi+z,\eta}(t) = v^{\xi,\eta}(t) +(1-e^{-t}) z$, so
\begin{equation*}
\begin{split}
&P_t^{\xi+z} F^\e(\cdot,\xi+z)- P_t^\xi F^\e(\cdot,\xi+z) -(1-e^{-t})\E\left(\displaystyle \int_D \dfrac{ \partial \alpha^\e} {\partial \xi} \left(v^{\xi,\eta}(t)\right)\cdot z \xi\cdot \phi dx\right)\\
=&\E \left(\displaystyle \int_D \left(\alpha^\e \left( v^{\xi,\eta}(t) + (1-e^{-t}) z\right) - \alpha^\e\left( v^{\xi,\eta}(t) \right)-\dfrac{ \partial \alpha^\e} {\partial \xi} \left(v^{\xi,\eta}(t) \right)\cdot (1-e^{-t})z  \right) \xi \cdot \phi dx \right).
\end{split}
\end{equation*}
We show that
\begin{equation}\label{a}\begin{split}
\lim_{z\to 0} & \E  \displaystyle \int_D \dfrac{\left(\alpha^\e \left( v^{\xi,\eta}(t) + (1-e^{-t}) z\right) 
- \alpha^\e\left( v^{\xi,\eta}(t) \right)-\dfrac{ \partial \alpha^\e} {\partial \xi} \left(v^{\xi,\eta}(t) \right)\cdot (1-e^{-t})z_n  \right) \xi\cdot \phi dx }{\| z\|_{L^2(D)^3}}  = 0.
\end{split}
\end{equation}

By contradiction assume that there exists a sequence $z_n\in L^2(D)^3$ with the norms converging to $0$, such that
\begin{equation*}\begin{split}
\lim_{n\to \infty} & \E  \displaystyle \int_D \dfrac{\left(\alpha^\e \left( v^{\xi,\eta}(t) + (1-e^{-t}) z_n\right) 
- \alpha^\e\left( v^{\xi,\eta}(t) \right)-\dfrac{ \partial \alpha^\e} {\partial \xi} \left(v^{\xi,\eta}(t) \right)\cdot (1-e^{-t})z_n  \right) \xi\cdot \phi dx }{\| z_n\|_{L^2(D)^3}}  \neq 0.
\end{split}
\end{equation*}
We can assume, by passing to a subsequence, that $z_n$ converges also pointwise to $0$.
But by the differentiability of the function $\alpha$, we have that
$$\dfrac{\left(\alpha^\e \left(v^{\xi,\eta}(t) + (1-e^{-t}) z_n\right) - \alpha^\e\left(v^{\xi,\eta}(t) \right)-\dfrac{ \partial \alpha^\e} {\partial \xi} \left(v^{\xi,\eta}(t) \right)\cdot (1-e^{-t})z_n  \right) \xi \cdot \phi}{\| z_n\|_{L^2(D)^3}}$$
converges to $0$ pointwise in $D$. We get a contradiction after applying  the dominated convergence theorem.

We use similar arguments and calculations to prove that

\begin{equation}\label{b}
\lim_{z\to 0} \displaystyle \dfrac{P_t^{\xi} F^\e (\cdot, \xi +  z) - P_t^{\xi} F^\e (\cdot, \xi)
-\displaystyle\int_{D}\E \left(\alpha^\e \left(v^{\xi,\eta}(t)\right) \phi\right)z dx}{\| z\|_{L^2(D)^3}} = 0,
\end{equation}
The limits \eqref{a} and \eqref{b} imply \eqref{psi2'}
which implies that
\begin{equation}
\label{estp2}
\left\| \dfrac{\partial}{\partial \xi} P_t^\xi F^\e(\cdot,\xi) (\eta)\right\|_{L^2(D)^3} \leq c\| \phi \|_{L^\infty(D)^3} (\|\xi\|_{L^2(D)^3} +1 ).
\end{equation}
\end{proof}
Now we compute $\displaystyle \dfrac {\partial} {\partial \xi}  \int_D \o{\alpha^\e}(\xi)\xi\phi dx = \displaystyle \dfrac {\partial} {\partial \xi}  \int_D \int_{L^2(D)^3}\alpha^\e(z)d\mu^\xi(z) \xi\phi dx $ and we show:
\begin{equation}\label{psi2''_2}
\dfrac{\partial}{\partial \xi} \int_D \o{\alpha^\e}(\xi)\xi\phi dx =\displaystyle \int_{L^2(D)^3} \dfrac { \partial \alpha^\e}{\partial \eta}(z)d\mu^\xi(z) \xi\phi  +  \int_{L^2(D)^3} \alpha^\e(z)d\mu^\xi(z) \phi.
\end{equation}
The measure $\mu^\xi$ having the mean $\xi$ and covariance $Q$ satisfies $\mu^\xi (B) =  \mu^0 (B+\xi)$ for every borel set $B \in L^2(D)^3$. So using the change of measure:
\begin{equation}
\begin{split}
\dfrac {\partial} {\partial \xi}  \int_D \int_{L^2(D)^3}\alpha^\e(z)d\mu^\xi(z) \xi\phi dx = \displaystyle \dfrac {\partial} {\partial \xi}  \int_D \int_{L^2(D)^3}\alpha^\e(z+\xi)d\mu^0(z) \xi\phi dx.
\end{split}
\end{equation}
Now \eqref{psi2''_2} follows using the same technique as for \eqref{a} and \eqref{b} and we also get:
\begin{equation}\label{estp3}
\left\|  \dfrac{\partial}{\partial \xi} \int_D \o{\alpha^\e}(\xi)\xi\phi dx \right\| \leq c\| \phi \|_{L^\infty(D)^3} (\|\xi\|_{L^2(D)^3} +1 ).
\end{equation}

We use the estimates \eqref{estp2} and \eqref{estp3} in the expression for $\Psi^\e$ and obtain \eqref{estpsi2}.

\subsection{Convergence of $S^\e_1$}
\label{subs43}
We start by applying It\^{o}'s formula to $\Psi^\e(v^\e(t), u^\e(t))$, for a fixed $\phi \in \bold{V} \cap L^\infty(D)^3$.
\begin{equation}
\begin{split}
\Psi^\e(v^\e(t), u^\e(t)) &=\Psi^\e(v^\e_0, u^\e_0)+ \int_0^t \dfrac{\partial \Psi^\e}{\partial \xi}(v^\e(s), u^\e(s))\cdot 
(\Delta u^\e(s) - \alpha^\e(v^\e(s))u^\e(s)+f(s))ds\\
&+ \int_0^t \dfrac{\partial \Psi^\e}{\partial \eta}(v^\e(s), u^\e(s))\cdot \dfrac{1}{\e}(- v^\e(s) + u^\e(s)) ds+ \dfrac{1}{\sqrt{\e}}\int_0^t \dfrac{\partial \Psi^\e}{\partial \eta}(v^\e(s), u^\e(s))\cdot \sqrt{Q} dW(s) \\
&+\dfrac{1}{2\e} \int_0^t Tr\left[Q \dfrac{\partial^2 \Psi^\e}{\partial v^2}(v^\e(s), u^\e(s))\right]ds\\
&=\Psi^\e(v^\e_0, u^\e_0)+ \int_0^t \dfrac{\partial \Psi^\e}{\partial \xi}(v^\e(s), u^\e(s))\cdot (\Delta u^\e(s) - \alpha^\e(v^\e(s))u^\e(s)+f(s))ds\\
&+\dfrac{1}{\e}\int_0^t \mathcal{L}^{u^\e(s)} \Psi^\e (v^\e(s), u^\e(s))ds\\
&+\dfrac{1}{\sqrt{\e}}\int_0^t \dfrac{\partial \Psi^\e}{\partial \eta}(v^\e(s), u^\e(s))\cdot \sqrt{Q} dW(s).
\end{split}
\end{equation}
We use the expression for the Kolmogorov operator $\mathcal{L}^\xi$ to get
\begin{equation}\nonumber
\begin{split}
\Psi^\e(v^\e(t), u^\e(t)) &= \Psi^\e(v^\e_0, u^\e_0)+ \int_0^t \dfrac{\partial \Psi^\e}{\partial \xi}(v^\e(s), u^\e(s))\cdot (\Delta u^\e(s) - \alpha^\e(v^\e(s))u^\e(s)+f(s))ds\\
&+\dfrac{c(\e)}{\e}\int_0^t  \Psi^\e (v^\e(s), u^\e(s))ds+\dfrac{1}{\e} \int_0^t \int_D \left(\alpha^\e(v^\e(s))ds - \o{\alpha^\e}(u^\e(s))\right)u^\e(s)\phi dx ds\\
&+\dfrac{1}{\sqrt{\e}}\int_0^t \dfrac{\partial \Psi^\e}{\partial \eta}(v^\e(s), u^\e(s))\cdot \sqrt{Q} dW(s)\Rightarrow
\end{split}
\end{equation}
\begin{equation}\label{int}
\begin{split}
\int_0^t \int_D &\left(\alpha^\e(v^\e(s)) - \o{\alpha^\e}(u^\e(s))\right)u^\e(s)\phi dx ds= \e \Psi^\e(v^\e(t), u^\e(t)) - \e \Psi^\e(v^\e_0, u^\e_0)  \\
&-\e \int_0^t \dfrac{\partial \Psi^\e}{\partial \xi}(v^\e(s), u^\e(s))\cdot (\Delta u^\e(s) - \alpha^\e(v^\e(s))u^\e(s)+f(s))ds \\
&-c(\e) \int_0^t  \Psi^\e (v^\e(s), u^\e(s))ds -\sqrt{\e}\int_0^t \dfrac{\partial \Psi^\e}{\partial \eta}(v^\e(s), u^\e(s))\cdot \sqrt{Q}dW(s).
\end{split}
\end{equation}
\begin{lemma}
\label{lemmaconv}
Let us define by $S^\e_1$ the integral $$\displaystyle\int_0^T  \int_D \left(\alpha^\e(v^\e(t)) - \o{\alpha^\e}(u^\e(t))\right)u^\e(t)\phi \psi'(t) dx dt.$$
Then for a particular choice of the sequence $c(\e)$ we have the following convergence:
\begin{equation}
\label{convre}
\begin{split}
\lim_{\e\to 0} \E\left| S^\e_1\right| =& 0.
\end{split}
\end{equation}
\end{lemma}
\begin{proof}
From the equation \eqref{int} we get:
\begin{equation}
\begin{split}
& \left|\int_0^t  \int_D \left(\alpha^\e(v^\e(s)) - \o{\alpha^\e}(u^\e(s))\right)u^\e(s)\phi dx ds\right|  \\
\leq &\e\left| \Psi^\e(v^\e(t), u^\e(t))\right| +\e\left| \Psi^\e(v^\e_0, u^\e_0)\right| + \sqrt{\e} \left|  \int_0^t \dfrac{\partial \Psi^\e}{\partial \eta}(v^\e(s), u^\e(s))\cdot  \sqrt{Q} dW(s)\right|\\
 +&c(\e)\int_0^t  \left| \Psi^\e (v^\e(s), u^\e(s)) \right| ds\\
+& \e \int_0^t \left|\dfrac{\partial \Psi^\e}{\partial \xi}(v^\e(s), u^\e(s))\cdot (\Delta u^\e(s) - \alpha^\e(v^\e(s))u^\e(s)+f(s))\right| ds.
\end{split}
\end{equation}
We use for the function $\Psi^\e$ the estimates \eqref{estpsi}, \eqref{estpsi1} and \eqref{estpsi2} and get:
\begin{equation}
\begin{split}
 &\left|\int_0^t  \int_D \left(\alpha^\e(v^\e(s)) - \o{\alpha^\e}(u^\e(s))\right)u^\e(s)\phi dx ds\right| \leq\\
 & c \e\|u^\e(t)\|_{L^2(D)^3} \|\phi\|_{L^\infty(D)^3} \left(1+\|v^\e(t)\|_{L^2(D)^3}\right) + c \e\|u^\e_0\|_{L^2(D)^3} \|\phi\|_{L^\infty(D)^3} \left(1+\|v^\e_0\|_{L^2(D)^3}\right)ds\\
+& \sqrt{\e} \left|  \int_0^t \dfrac{\partial \Psi^\e}{\partial \eta}(v^\e(s), u^\e(s))\cdot \sqrt{Q} dW(s)\right|\\
+&c(\e)\int_0^t  c \e\|u^\e(s)\|_{L^2(D)^3} \|\phi\|_{L^\infty(D)^3} \left(1+\|v^\e(s)\|_{L^2(D)^3}\right) ds\\
+& \e   \int_0^t    \left\|\dfrac{\partial \Psi^\e}{\partial \xi}(v^\e(s), u^\e(s))\right\|_{L^2(D)^3} \cdot \left\|(\Delta u^\e(s) - \alpha^\e(v^\e(s))u^\e(s)+f(s)) \right\|_{L^2(D)^3}\\
\leq& c\e  \|\phi\|_{L^\infty(D)^3} \sup_{t\in[0,T]} \|u^\e(t)\|_{L^2(D)^3} (1+ \|v^\e(t)\|_{L^2(D)^3} +\|v^\e_0\|_{L^2(D)^3})\\
+ &\sqrt{\e} \left|  \int_0^t \dfrac{\partial \Psi^\e}{\partial \eta}(v^\e(s), u^\e(s))\cdot\sqrt{Q}dW(s)\right|\\
+&c \e c(\e) \|\phi\|_{L^\infty(D)^3} \sup_{t\in[0,T]} \|u^\e(t)\|_{L^2(D)^3} \int_0^t \left(1+\|v^\e(s)\|_{L^2(D)^3}\right) ds\\
+& \e   \int_0^t   \dfrac{c}{c(\e)}\left(1+ \|u^\e(s)\|_{L^2(D)^3}\right)\| \phi \|_{L^\infty(D)^3}\cdot \left(\left\|(u^\e(s)\right\|_{H^2(D)^3} +\left\|f(s)) \right\|_{L^2(D)^3}\right).\\
\end{split}
\end{equation}
But, It\^{o}'s isometry gives us that
\begin{equation}\nonumber
\begin{split}
\E \left|  \int_0^t \dfrac{\partial \Psi^\e}{\partial \eta}(v^\e(s), u^\e(s))\cdot \sqrt{Q}dW(s)\right| &\leq c \left(\int_0^t \left\| \dfrac{\partial \Psi^\e}{\partial \eta}(v^\e(s), u^\e(s)) \right\|^2_{L^2(D)^3}\right)^{1/2}\\ 
&\leq \dfrac{c}{\sqrt{c(\e)}}\| \phi \|_{L^\infty(D)^3} \left( \int_0^t\left(1+ \|u^\e(s)\|_{L^2(D)^3}\right)^2 ds \right)^{1/2},
\end{split}
\end{equation}
So, given the estimates for $u^\e$, we obtain that
\begin{equation}\nonumber
\begin{split}
\E \left| S^\e_1 \right| &\leq C_T \e \E (1+ \sup_{t\in[0,T]}\|v^\e(t)\|_{L^2(D)^3} ) +C_T \sqrt{\dfrac{\e}{c(\e)}} + C_T\dfrac{\e}{c(\e)}.
\end{split}
\end{equation}
Now using the estimate for $v^\e$ and choosing $c(\e) =\sqrt{\e}$ we get \eqref{convre}.
\end{proof}

\subsection{Convergence of $S^\e_3$}
\label{subs44}
\begin{lemma}
\label{lemmaconv'}
For fixed $\o{u}\in L^\infty (\Omega; C([0,T]; L^2(D)^3))$ and $\phi\in L^\infty(D)^3$ let us define by $S^\e_3$ the integral $\displaystyle \int_0^T \int_D \left(\o{\alpha^{\e}}(\o{u}(s))\o{u}(t) - \o{\alpha}(\o{u}(t))\o{u}(t)\right)\phi \psi (t)dxdt$. Then:
\begin{equation}
\label{convre}
\begin{split}
\lim_{\e\to 0} \E\left| S^\e_3\right| =& 0.
\end{split}
\end{equation}
\end{lemma}
\begin{proof}
For any $t\in[0,T]$ consider the sequence of functions $F^\e_t: L^2(D)^3 \to L^2(D)^3$, 
$$F^\e_t(z) (x) = \left(\alpha \left(\dfrac{x}{\e}, z(x)\right)  -  \int_Y\alpha \left(y, z(x)\right) \right) u(t,x).$$
We show now that for any $z\in L^2(D)^3$, for every $t\in[0,T]$ and a.e. $\omega\in\Omega$, $F^\e_t (z)$ converges in $L^2(D)^3$ to $0$. We we fix $\omega$ and $t$ and let $z_n$ and $w_n$ two sequences of continuous functions converging in $L^2(D)^3$ to $z$ and $u(t)$. We use Lemma 1.3 from \cite{A-2s} and obtain that the sequence $ F^\e_n(x) = \left(\alpha \left(\dfrac{x}{\e}, z_n(x)\right)  -  \displaystyle\int_Y\alpha \left(y, z_n(x)\right) \right) w_n(x)$ converges when $\e\to 0$ to $0$ in $L^2(D)^3$.

But $$\left|F^\e_n(x) - F^\e_t(z)(x) \right| \leq c |w_n(x)- u(t,x)| + c|z_n(x)-z(x)|,$$
based on the Lipschitz condition and boundedness for $\alpha$. 
We deduce that $F^\e_t (z)$ converges in $L^2(D)^3$ to $0$. The sequence being also uniformly bounded by $\|\o{u} \|_{L^\infty (\Omega; C([0,T]; L^2(D)^3))}$, Vitali's convergence theorem implies that the sequence of the integrals with respect to the probability measure on $L^2(D)^3$, $\mu^{\o{u}(t)}$ also converge to $0$ in $L^2(D)^3$:
$$\lim_{\e\to 0}\int_{L^2(D)^3} F^\e_t (z) d\mu^{\o{u}(t)} dz = 0 \ in \  L^2(D)^3,$$
which can be rewritten as
$$\lim_{\e\to 0} \o{\alpha^\e}(\o{u}(t)) \o{u}(t) - \o{\alpha}(\o{u}(t)) \o{u}(t) = 0  \ in \  L^2(D)^3.$$
This implies that $\mathbb{P}$ a.s. and for every $t\in[0,T]$
$$\lim_{\e\to 0} \int_D \left(\o{\alpha^\e}(\o{u}(t)) \o{u}(t)- \o{\alpha}(\o{u}(t)) \o{u}(t)\right)\phi \psi' (t)dx = 0,$$
with the sequence being also uniformly bounded by $c \|\o{u}\|_{L^\infty(\Omega, C([0,T];L^2(D)^3))} \| \phi\|_{L^\infty(D)^3} \| \psi '\|_{L^\infty[0,T]}$.  We apply the bounded convergence theorem and integrate over $\Omega \times [0,T]$ to get the result.
\end{proof}
\subsection{Proof of Theorem \ref{thconv}}
\label{subs45}
\begin{proof}
The uniform bounds \eqref{est1'}, \eqref{est2'} and \eqref{est3'} hold for $u^\e$. Then, for a.e. $\omega \in \Omega$ the sequence is included in a compact set $\mathcal{K}$ of $L^2(0,T;\bold{V})$ and  the sequence is tight. Then, there exists a subsequence $u^{\e'}$ and a random element $\o{u} \in L^2(0,T;\bold{V})$ such that $u^{\e'}$ converges in distribution to $\o{u}$. Skorokhod theorem gives us the existence of another subsequence $u^{\e''}$ and a sequence $\wt{u^{\e''}}$ with the same distribution defined on another probability space $\wt{\Omega}$ that converges pointwise to some $\wt{\o{u}}$, a random element of $L^2(0,T;\bold{V})$ with the same distribution as $\o{u}$. It follows from here that $\wt{\o{u}}$ belongs to $\mathcal{K}$ a.s. so $\wt{\o{u}}\in L^\infty(\wt{\Omega}, L^2(0,T;\bold{V}))$ and $\o{u}\in L^\infty(\Omega, L^2(0,T;\bold{V}))$.
We use the variational formulation \eqref{weaksole} for $u^{\e''}$ with a test function $\phi \in \bold{V}\cap L^\infty(D)^3$, multiply it with $\psi'$, where $\psi \in C^1[0,T]$ with $\psi(T)=0$ to get:
\begin{equation}
\label{weaksole'}
\begin{split}
&\int_0^T \int_D u^{\e''}(t) \phi \psi'(t) dx dt - \int_0^T\int_D u^{\e''}_0 \phi \psi'(t) dx dt - \int_0^T \int_D \nabla u^{\e''} (t) \nabla \phi \psi(t)dx dt\\
 -&\int_0^T \int_D \alpha^{\e''}(v^{\e''}(t)) u^{\e''}(t) \phi \psi(t)dx dt = -\int_0^T \int_D f(t) \phi \psi(t) dx dt.
\end{split}
\end{equation}
We now show that
\begin{equation}
\label{convre'}
\lim_{\e''\to 0} \E\left|\displaystyle\int_0^T  \int_D \left(\alpha^{\e''}(v^{\e''}(t))u^{\e''}(t) - \o{\alpha}(\o{u}(t))\o{u}(t)\right)\phi \psi(t)dx dt \right| = 0.
\end{equation}
We write:
\begin{equation}\nonumber
\begin{split}
\alpha^{\e''}(v^{\e''})u^{\e''} - \o{\alpha}(\o{u})\o{u} &= \alpha^{\e''}(v^{\e''})u^{\e''} - \o{\alpha^{\e''}}(u^{\e''})u^{\e''} + \o{\alpha^{\e''}}(u^{\e''})u^{\e''}-\o{\alpha^{\e''}}(\o{u})u^{\e''}\\
&+\o{\alpha^{\e''}}(\o{u})u^{\e''} -\o{\alpha^{\e''}}(\o{u})\o{u}+ \o{\alpha^{\e''}}(\o{u})\o{u} -\o{\alpha}(\o{u})\o{u},
\end{split}
\end{equation}
so
\begin{equation}\nonumber
\begin{split}
\int_0^T \int_D \alpha^{\e''}(v^{\e''}(t))u^{\e''}(t)\phi \psi(t) dx dt = S^{\e''}_1 + S^{\e''}_2  +S^{\e''}_3,
\end{split}
\end{equation}
where
\begin{equation}
\label{s1}
S^\e_1  = \int_0^T  \int_D \left(\alpha^\e(v^\e(t)) - \o{\alpha^\e}(u^\e(t))\right)u^\e(t)\phi \psi(t) dxdt,
\end{equation}
\begin{equation}
\label{s2}
S^\e_2  = \int_0^T \int_D \left( \o{\alpha^\e}(u^\e(t))u^\e(t) - \o{\alpha^{\e}}(\o{u}(t))\o{u}(t) \right)\phi \psi(t)dx dt,
\end{equation}
and
\begin{equation}
\label{s3}
S^\e_3  = \int_0^T  \int_D \left(\o{\alpha^{\e}}(\o{u}(t))\o{u}(t) - \o{\alpha}(\o{u}(t))\o{u}(t)\right)\phi \psi (t)dxdt.
\end{equation}
Given the same distribution for the sequences $u^\e$ and $\wt{u^\e}$, for $\o{u}$ and $\wt{\o{u}}$ and the uniform bounds in $L^2(0,T;\bold{V})$ we have:
\begin{equation*}
\begin{split}
\E \left| S^{\e''}_2\right| &= \E \left| \int_0^T \int_D \left( \o{\alpha^{\e''}}(\wt{u^{\e''}}(t))\wt{u^{\e''}}(t) - \o{\alpha^{\e''}}(\wt{\o{u}}(t))\wt{\o{u}}(t) \right)\phi \psi(t)dx dt\right|\\
&\leq \E \left| \int_0^T \int_D \left( \o{\alpha^{\e''}}(\wt{u^{\e''}}(t))\wt{u^{\e''}}(t) - \o{\alpha^{\e''}}(\wt{u^{\e''}}(t))\wt{\o{u}}(t) \right)\phi \psi(t)dx dt\right|\\
&+  \E \left| \int_0^T \int_D \left( \o{\alpha^{\e''}}(\wt{u^{\e''}}(t))\wt{\o{u}}(t) - \o{\alpha^{\e''}}(\wt{\o{u}}(t))\wt{\o{u}}(t) \right)\phi \psi(t)dx dt\right|.
\end{split}
\end{equation*}
The function $\o{\alpha^{\e}}$ is Lipschitz, with the same constant as $\alpha$ and also bounded so we get:
\begin{equation}\nonumber
\begin{split}
\E \left| S^{\e''}_2\right| \leq c\|\phi\|_{L^\infty(D)^3} \| \psi \|_{L^\infty[0,T]} \E \left\| \wt{u^{\e''}} - \wt{\o{u}}\right\|_{L^1(0,T;\bold{V})}.
\end{split}
\end{equation}
The pointwise convergence of $\wt{u^{\e''}}$ to $\wt{\o{u}}$ in $L^2(0,T;\bold{V})$ and the uniform bounds imply that 
\begin{equation}\nonumber
\lim_{\e''\to 0} \E \left| S^{\e''}_2\right| = 0.
\end{equation}
This, together with the limits given by the Lemmas \ref{lemmaconv} and \ref{lemmaconv'} give \eqref{convre'}. \eqref{convre'} and \eqref{weaksole'} imply that
\begin{equation}\nonumber
\E \left|  \int_0^T \int_D \left(u^{\e''}(t)\phi\psi'(t) - u_0^{\e''} \phi\psi' (t) -\nabla u^{\e''}(t) \nabla \phi \psi(t) + f(t)\phi\psi(t) - \o{\alpha} (\o{u}(t)) \o{u}(t) \phi\psi(t)\right)dx dt \right | \to 0.
\end{equation}
Now, using the fact that $u^\e$ and $\wt{u^\e}$ are equally distributed implies that

\begin{equation}\nonumber
\E \left|  \int_0^T \int_D \left(\wt{u^{\e''}}(t)\phi\psi'(t) - u_0^{\e''} \phi\psi' (t) -\nabla \wt{u^{\e''}}(t) \nabla \phi \psi(t) + f(t)\phi\psi(t) - \o{\alpha} (\wt{\o{u}}(t)) \wt{\o{u}}(t) \phi\psi(t)\right)dxdt \right | \to 0.
\end{equation}
The sequence above converges in $L^1(\wt{\Omega})$ to $0$ but also pointwise in $\wt{\Omega}$ to 
$$\int_0^T \int_D \left(\wt{\o{u}}(t)\phi\psi'(t) - u_0 \phi\psi' (t) -\nabla \wt{\o{u}}(t) \nabla \phi \psi(t) + f(t)\phi\psi(t) - \o{\alpha} (\wt{\o{u}}(t)) \wt{\o{u}}(t) \phi\psi(t)\right)dxdt, $$
which means that $\wt{\o{u}}$ is pointwise the weak solution of the deterministic equation \eqref{eqou} which, according to Theorem \ref{thexunou} has a unique solution, so $\wt{\o{u}}$ and $\o{u}$ are deterministic. Then, the whole sequence $u^\e$ converges to $\o{u}$ in distribution, and since $\o{u}$ is deterministic then the convergence is also in probability  see \cite{JP} Theorem 18.3.
\end{proof}

\section*{Acknowledgements}
 Hakima Bessaih is  supported by the NSF grants DMS-1416689 and DMS-1418838.
Yalchin Efendiev's work is partially supported by the U.S. Department of Energy Office of Science, Office of Advanced Scientific Computing Research, Applied Mathematics program under Award Number DE-FG02-13ER26165 and  the DoD Army ARO Project.


\end{document}